\documentclass[oneside,english,10pt]{amsart}
\usepackage{color}
\usepackage{enumerate}
\usepackage{graphicx}

\newcommand{\G}{\mathcal{G}}

\usepackage{amsmath}

\usepackage{amsmath,amssymb,amsfonts,amsthm}
\newtheorem{theorem}{Theorem}[section]
\newtheorem{maintheorem}{Theorem}

\newtheorem{secondtheorem}{Theorem}

\newtheorem{thirdtheorem}{Theorem}

\newcommand{\sing}{{\rm Sing}}

\newtheorem{proposition}[theorem]{Proposition}

\newtheorem{lemma}[theorem]{Lemma}
\theoremstyle{definition}

\newtheorem{example}[theorem]{Example}
\newtheorem{remark}[theorem]{Remark}

\newcommand{\F}{\mathcal{F}}
\newcommand{\C}{\mathbb{C}^2}

\def \C {\mathbb{C}}
\def \calc {{\mathcal C}}

\def \DC {\widetilde{\Delta}}

\def \ddd {\mathcal{D}}
\def \dps {\displaystyle}

\def \fol {{\mathcal F}}
\def \F {{\mathcal F}}

\def \G {{\mathcal G}}
\def \folt {\widetilde{\mathcal{F}}}

\def \ll {\mathcal{L}}

\def \oo {\mathcal{O}}
\def \aa {\mathcal{A}}

\def \P {\mathbb{P}}
\def \pn {\mathbb{P}^n}
\def \pnt {\widetilde{\mathbb{P}}^n}

\def \sing {{\rm Sing}}

\def \singf {{\rm Sing}(\mathcal{F})}
\def \dim {{\rm dim}}
\def \tt {\mathcal{T}}
\def \cod {{\rm cod}}

\def \Z {{\mathbb Z}}

\begin{document}
\title{On foliations by curves with singularities of positive dimension}
\date{\today}
\author{Arturo Fern\' andez-P\' erez and Gilcione Nonato Costa}
\address{Departamento de Matem\'atica - ICEX, Universidade Federal de Minas Gerais, UFMG}
\curraddr{Av. Ant\^onio Carlos 6627, 31270-901, Belo Horizonte-MG, Brasil.}
\email{fernandez@ufmg.br}
\email{gilcione@ufmg.br}
\subjclass[2010]{Primary 32S65; Secondary 58K45}
\keywords{Foliations by curves, non-isolated singularities, Milnor number}

\begin{abstract}
We present enumerative results for holomorphic foliations by curves on $\pn$, $n\geq 3$, with singularities of positive dimension. 
Some of the results presented improve previous ones due to Corr\^ea--Fern\'andez-P\'erez--Nonato Costa--Vidal Martins \cite{correa} and Nonato Costa \cite{toulouse}. We also present an enumerative result  bounding the number of isolated singularities in a projective subvariety invariant by a holomorphic foliation by curves on $\pn$ with a singularity of positive dimension.   
\end{abstract}
\maketitle


\section{Introduction and statement of results}

\par The purpose of the present paper is to establish numerical relations between the invariants of a foliation by curves $\fol$
on the complex projective space $\pn$ and the characteristic classes
of the irreducible components of its singular set $\sing(\fol)$. By
numerical invariants of a foliation $\fol$, we refer to Chern
numbers of its tangent sheaf, the order of vanishing and the
multiplicity of $\fol$ along its singularities. Since any foliation
by curves on $\pn$ with singular set of positive dimension can be
deformed in to a foliation by curves with isolated singularities
\cite{bound}, we address the problem of determining, in terms of
Milnor number, the numerical contribution that each irreducible
component  of the singular set offers after being deformed.

\par More precisely, a holomorphic foliation by curves $\fol$ of $\pn$, $n\geq 2$, of degree $k\geq1$, is defined  by a morphism
$$\Phi:\mathcal{O}_{\pn}(1-k)\to T\pn$$
whose singular set $\sing(\fol):=\{p:\Phi(p)=0\}$ has codimension at
least two. Outside the singular set of $\fol$, the morphism $\Phi$
induces an integrable distribution $\mathcal{D}$ of tangent
directions. By definition, the leaves of $\fol$ are the leaves
induced by $\mathcal{D}$ on $\pn\setminus\sing(\fol)$. We call
$\ll_{\fol}:=\mathcal{O}_{\pn}(1-k)$ the \textit{tangent sheaf} of
$\fol$. It follows from the definition that $\fol$ can be
represented in affine coordinates $(z_1,...,z_n)$ by a vector field
of the form
$$X(z)=\sum^{k}_{j=0}X_j(z)+g(z)\cdot R(z)$$
where $R(z)=\displaystyle\sum^{n}_{i=1}z_{i}\frac{\partial}{\partial z_{i}}$,
$g\in\mathbb{C}[z_1,\ldots,z_n]$ is homogeneous of degree $k$ and
$X_{j}$, $0\leq j\leq k$, is a vector field whose components are
homogenous polynomials of degree $j$.
\par Let $W$ be a projective subvariety of $\pn$ of codimension $d\geq 2$. Let $\widetilde{\P}^n$ be the blowup of $\pn$ along $W$, and $\pi:\widetilde{\P}^n\rightarrow \pn$
the blowup morphism with exceptional divisor $E:=\pi^{-1}(W)$. Here
and subsequently $\widetilde{\fol}:=\pi^{*}\fol$ denotes the strict
transformed of $\fol$ by $\pi$ and  $\ell$ denotes the order of
vanishing of $\widetilde{\fol}$ along $E$. It is easy to check that
tangent sheaf of $\widetilde{\fol}$ is
$$\mathcal{L}_{\tilde{\fol}}=\pi^{*}(\mathcal{L}_{\fol})\otimes\mathcal{O}_{\widetilde{\P}^n}(\ell
E).$$
\par A holomorphic  foliation by curves $\fol$ on $\pn$ it said to be {\it special } along $W\subset\sing(\fol)$, if
$E$ is an invariant set of $\widetilde{\fol}$, and ${\rm
Sing}(\widetilde{\fol})$ meets $E$ in at most isolated singularities
. Special foliation along curves has been studied for $n=3$ in
\cite{toulouse} and for $n>3$ in \cite{correa} and \cite{indices}.
As we will see below, the notion of speciality of a foliation is
fundamental to the study of the irreducible components of the
singular set of a foliation by curves.
\par In what follows $\mathcal{W}^{(d)}_{\delta}$ and $\sigma_i^{(d)}$ are
the Wronski (or complete symmetric) and elementary symmetric
functions of degree $\delta$ in $d$ variables, respectively. That is

$$\mathcal{W}^{(d)}_{\delta}:=\mathcal{W}^{(d)}_{\delta}(k_1,\ldots,k_{d})=\sum_{i_{1}+\ldots+i_{d}=\delta}k^{i_{1}}_{1}\ldots k^{i_{d}}_{d}$$
and
$$\sigma^{(d)}_i:=\sigma_i^{(d)}(k_1,\ldots,k_d)=\dps\sum_{1\le
j_{1}<\ldots<j_{i}\le d}k_{j_1}\cdots k_{j_i}.$$
\par When $W$ is a smooth intersection complete on $\pn$ of multidegree $(k_1,\ldots,k_d)$, we will use the symbol $\nu(\fol,W)$ to denote
\begin{align}\label{nu_equation}
\nu(\fol,W)=-\deg(W)\sum_{|a|=0}^{n-d}\sum_{m=0}^{n-d-|a|}(-1)^{\delta_{|a|}^{m}}\frac{\varphi_{a}^{(m)}(\ell)}{m!}(\deg(\fol-1)^{m}\sigma_{a_1}^{(d)}\tau_{a_2}^{(d)}\mathcal{W}_{\delta_{|a|}^{m}}^{(d)},
\end{align}
where $a=(a_1,a_2)$, $0\le a_1 \le d$, $0\le a_2 \le n-d$, $|a|=a_1+a_2$,
$\delta_{|a|}^{m}:=n-d-|a|-m$,
\begin{align*}
\varphi_{a}(x)=x^{n-d-a_2}(1+x)^{d-a_1}\,\,\,\text{and}\,\,\,\,
\tau_i^{(d)}=\dps\sum_{j=0}^{i}(-1)^j{{n+1}\choose{i-j}}\mathcal{W}_j^{(d)}(k_1,\ldots,k_d).
\end{align*}
\par We will show that $\nu(\fol,W)$ is strongly related with the numerical contribution that $W$ offers after a small deformation of $\fol$. To complete our notation, we introduce the notion of
\textit{embedded close points} of an irreducible component of $W$ of
$\sing(\fol)$. More precisely, for an irreducible component $W$ of
$\sing(\fol)$, we will prove that it is possible to deform $\fol$
into a foliation by curves $\fol_t$ (for some $0<t<\epsilon$) on
$\pn$ such that either $W$ is invariant by $\fol_t$ or $\fol_t$ is
special along $W$ (see Lemma \ref{lema_Theorem1}). The set of
isolated points $\{p_i^t,i=1,\ldots,s_t\}\subset \sing(\fol_t)$ such
that $p_i^t\notin W$ but $\displaystyle\lim_{t\to0}p_i^t \in W$ will
be called \textit{embedded close points} associated to $W$ and we
will be denoted by $\mathcal{A}_W$.
\par With the above notation, we state the main result.

\begin{maintheorem}\label{theorem1}
Let $\fol$ be a holomorphic foliation by curves on $\pn$, $n\geq 3$,
of degree $k$. Suppose that its singular set $\sing(\fol)$ is the
disjoint union of smooth scheme-theoretic complete intersection
subvarieties $W_1,\ldots,W_r$ of pure codimension $d_j\geq 2$,
respectively, and points $p_1,\ldots,p_s$. Then
 $$\dps\sum_{i=1}^{s}\mu(\fol,p_i)=\dps\sum_{i=0}^{n}k^i+\dps\sum_{i=1}^r
\nu(\fol,W_i)-\sum_{i=1}^{r}N(\fol,\aa_{W_i}),$$
where $N(\fol,\aa_{W_i})$ denotes the number (counted with multiplicity) of embedding closed points associated to each $W_i$.
\end{maintheorem}
\par It follows easily that  the numerical contribution of each $W_i$ is exactly  $N(\fol,\aa_{W_i})-\nu(\fol,W_i)$. In this sense, it can be interpret it as the Milnor number of $\fol$ at $W_i$,
\begin{equation}
\mu(\fol,W_i):=N(\fol,\aa_{W_i})-\nu(\fol,W_i).
\end{equation}
\par In general, for an arbitrary foliation with non-isolated singularities on $\pn$ it seems to be difficult to determine effectively $N(\fol,\aa_{W_i})$.
In the final section of the paper, we present a family of foliations
by curves on $\pn$ where it is possible to obtain an upper bound to
this number.
\par On the other hand, to prove Theorem \ref{theorem1} we need prove the following result.

\begin{secondtheorem}\label{second}
Let $\fol$ be a holomorphic foliation by curves on $\pn$, $n\geq 3$,
of degree $k$. Suppose that its singular set $\sing(\fol)$ is the
disjoint union of smooth scheme-theoretic complete intersection
subvarieties $W_1,\ldots,W_r$ of pure codimension $d_j\geq 2$,
respectively, and points $p_1,\ldots,p_s$. If $\fol$ is special
along each $W_i$.
Then
$$ \dps\sum_{i=1}^{s}\mu(\fol,p_i) =
\dps\sum_{i=0}^{n}k^i+\dps\sum_{i=1}^r
\nu(\fol,W_i).
$$
 In particular, $$\mu(\fol,W_i)=-\nu(\fol,W_i).
$$

\end{secondtheorem}
\par  Note that the speciality  condition implies that the set of embedding closed points of subvarieties $W_i$ is empty, for all $i=1,\ldots,r$.
Moreover, in this case, $\mu(\fol,W_i)$ is explicitly determined and
depends of the invariants of $\fol$ and $W$. Since we work with
irreducible components of $\sing(\fol)$ of any dimension, it is
clear that Theorem \ref{second} generalizes Theorem 1 of
\cite{correa} and the main result of \cite{toulouse}.
\par Now, in the same spirit, the paper also addresses the following question: given a foliation by curves $\fol$ on $\pn$, $n\geq 3$, and a smooth algebraic subvariety $V$ of $\pn$ invariant by $\fol$.
Assuming that the restriction foliation $\fol|_{V}$ to $V$ has only
an irreducible component $W$ of positive dimension and some isolated
points of $\singf$ contained in $V$. How many isolated singularities
of $\fol$, counted with multiplicity, are there on $V$?
\par The above question was completely solved by Marcio Soares \cite{soares} in the case that $W$ is empty, that is, when $\fol$ only admits isolated singularities on $V$.
It is known from \cite{soares} that the number of isolated
singularities of $\fol$ in $V$ only depends of invariants of $\fol$
and $V$. The results of \cite{soares} were applied to solve an
important question, originally raised by Henri Poincar\'e
\cite{poincare}, of bounding the degree of an irreducible plane
algebraic curve invariant by a holomorphic foliation $\fol$ of
$\mathbb{P}^2$, in terms of the degree of the foliation. For more
details about this important problem we refer to a \cite{inventions}
and the references within. Of course, in the case $W\neq\emptyset$,
our expectation is proving a similar result.

\par Let $V=Z(f_1,\ldots,f_m)$ be a smooth complete intersection on $\pn$. Let $\widetilde{\fol}$ be the strict transform of $\fol$ by the blowup $\pi$ at $W$ with exceptional divisor $E$. Assume that $W\subset V$.
Let $\widetilde{V}=\overline{\pi^{-1}(V\setminus W)}$ and
$V_E=\widetilde{V}\cap E$. Set $N:=N_{W/\pn}$,
$\zeta:=c_1(\mathcal{O}_N(-1))$ and $\mathbf{h}$ the class of a
hyperplane of $\pn$. If $W=Z(h_1,\ldots,h_d)$ is a smooth complete
intersection on $\pn$ of multidegree $(k_1,\ldots,k_d)$, we define
$$\tau_i^{(m)}:=\tau_i^{(m)}(d_1,\ldots,d_m)=\dps\sum_{j=0}^{i}(-1)^j{{n+1}\choose{i-j}}\mathcal{W}_j^{(m)}(d_1,\ldots,d_m).$$
with $d_j=\deg(f_j)$ and also define the following
\begin{align*}
\nu(\fol,V,W)&=\dps\sum_{i=0}^{n-d}\sum_{j=0}^{n-d-i}\frac{\Omega^{(j)}(\ell)}{j!}(-1)^{\gamma_i^j}\tau_i^{(m)}(k-1)^j\alpha_{\widetilde{V}_E}^{\gamma_i^j-1}
\end{align*}
where $\gamma_i^j=n-m-i-j$,
$\alpha_{{V}_E}^{(i)}:=\dps\int_{\widetilde{V}_E}\pi^*(\mathbf{h})^{n-m-1-i}\cdot
\zeta^i$ and $\Omega(x)=\dps\sum_{p=i}^{n-m-j}x^{n-m-p}$.
\par Let $\mathcal{I}(W):=\{g\in\mathbb{C}[x_0,\ldots,x_n]:g=\lambda_1h_1+\ldots+\lambda_d h_{d},\quad\text{for some} \quad(\lambda_1,\ldots,\lambda_d)\in\C^{d}\}.$
With the above notation, we give a partial solution to the above
question in a very particular case.
\begin{thirdtheorem}\label{principal}
Let $\fol$ be a foliation by curves on $\P^n$, $n\ge 3$, of degree $k$, such that
$$
 \mbox{Sing}(\fol) = W \cup\{p_1,\ldots,p_s\},
$$
where $W=Z(h_{1},\ldots,h_{d})$ is a smooth complete intersection of $\pn$ and $p_j$, $1\leq j\leq s$, are isolated points disjoint to $W$.
Suppose that  $V=Z(f_1,\ldots,f_m)$ is a smooth irreducible complete intersection of $\pn$, invariant by $\fol$ and such that $f_i\in\mathcal{I}(W)$ for all $1\leq i\leq m$. Then $W\subset V$ and

$$ \sum_{p\in V\setminus W}\mu(\fol,p)\le N(\fol,V)+\nu(\fol,V,W)+N(\fol,A_{V\setminus W}),$$
where $N(\fol,A_{V\setminus W})$ is the number, counted with
multiplicity, of embedding closed points of $V\setminus W$.
\end{thirdtheorem}
\par The proof of Theorem \ref{principal} is similar to the proof of Theorem \ref{theorem1}. More precisely, we construct a foliation
by curves $\fol_t$ of $\pn$ that is a deformation of $\fol$ such
that $V$ is invariant and either $W$ is invariant by $\fol_t$ or
$\fol_t$ is special along $W$, with some isolated singularities.
Furthermore, we will show that the total sum of Milnor numbers of
$\fol_t$ associated to each singularity is independent of the
parameter $t$.

\section{Preliminaries}

\par In this section we focus on some basic definitions and known results of holomorphic foliations by curves on $\pn$.
 \subsection{Milnor number of an isolated singularity}
Let $\fol$ be a holomorphic foliation by curves on $\pn$ and $p$ a
singular point of $\fol$. Without loss of generality we can suppose
that $p=0\in\C^n$ and $\fol$ is defined, in a neighborhood of $p$,
by a polynomial vector field
$X_p=\displaystyle\sum^{n}_{j=1}P_{j}(z)\frac{\partial}{\partial{z_{n}}},$ where
$P_i\in\C[z_1,\ldots,z_n]$. The \textit{Milnor number} of $\fol$ at
$p$, denoted by $\mu(\fol,p)$, is the $\C$-dimension of the local
algebra associated to $X_p$,
$$\mu(\fol,p)=\dim_{\mathbb{C}}\frac{\mathcal{O}_{n,0}}{\displaystyle \langle P_1,\ldots,P_n\rangle}$$
where $\mathcal{O}_{n,0}$ is the ring of germs of holomorphic
functions at $0\in\C^n$.  It is well-known that this definition is
independent of the vector field representing $\fol$ at $p$.
\par We say that $\fol$ is \textit{non-degenerate} on $\pn$ if $\fol$ has only isolated singularities and the Milnor number of a vector field defining the
foliation around each singularity is 1. The Jouanolou's foliation $\mathcal{J}_k$ of degree $k\geq 1$, given by the vector field  in affine coordinates $(z_1,\ldots,z_n)$
\begin{equation}\label{jouanolou}
\sum_{i=1}^{n}\Big(\sum^{k}_{j=1}z_iz_1^{k}-z^{k}_{i+1}\Big)\frac{\partial}{\partial{z}_i}+(z_nz_1^k-1)\frac{\partial}{\partial{z}_n}
\end{equation}
 is an example of a non-degenerate foliation on $\pn$. Now if $\fol$ is non-degenerate foliation on $\pn$ then it follows from the classical Baum-Bott formula \cite{BB} that $$\sum_{p\in\pn}\mu(\fol,p)=k^n+k^{n-1}+\ldots+k+1.$$

\subsection{Multiplicity along subvarieties}
We describe the multiplicity of $\F$ along an irreducible variety
$W$ of pure codimension $\cod(W)=d$, with $d\geq 2$.  From now on we
make the assumption $W\subset\singf$. By a holomorphic change of
coordinates, $W$ can be locally given as $\{z_1=\ldots=z_{d}=0\}$.
In this neighborhood, $\fol$ is described by the vector field
\begin{equation}
\label{fol} \ddd_{\fol}=\dps P_1(z)\frac{\partial}{\partial
z_1}+\ldots+\dps P_n(z)\frac{\partial}{\partial z_n}.
\end{equation}
Therefore, one may write the local sections as
\begin{equation}
\label{for40}
        P_i(z) = \dps\sum_{|a|=m_i}z_1^{a_1}\cdots z_{d}^{a_{d}}P_{i,a}(z)
\end{equation}
where $a:=(a_1,\ldots,a_{d})\in \Z^d$ with $|a|:=a_1+\ldots+a_{d}$,
$a_i\ge0$ and at least one among the $P_{i,a}(z)$ does not vanish at
$\{z_{1}=\ldots=z_{d}=0\}$. We define the multiplicity of each
function $P_i$ along $W$ as $m_{W}(P_i)=m_i$. As a consequence, the
\textit{multiplicity of $\fol$ along $W$} is defined as
\begin{equation}
\label{eqummm}
m_{W}(\fol)={\rm min}\{m_1,\ldots,m_n\}.
\end{equation}

\par In order to simplify our analysis, we use the following Lemma:
\begin{lemma}\label{lemafp} Let $\fol$ be a holomorphic foliation by curves,
defined in the neighborhood of $p$ by the polynomial vector field
$$ X_p = \displaystyle\sum^{n}_{j=1}P_{j}(z)\frac{\partial}{\partial{z_{j}}}$$
where $m_i=m_{W}(P_i)$. Then, by a linear change of coordinates,
$\fol$ may be described by the polynomial vector field

$$ Y_p = \displaystyle\sum^{n}_{j=1}Q_{j}(w)\frac{\partial}{\partial{z_{j}}}$$
with $$m_{W}(Q_j)=\left\{\begin{array}{ll}
                            m^{\prime}_{W}(\fol),&{\rm for}\quad j=1,\ldots,d\cr
                            m_{W}(\fol),&{\rm for}\quad j=d+1,\ldots,n
                            \end{array}\right.$$
where $m^{\prime}_{W}(\fol)={\rm min}\{m_1,\ldots,m_d\}$.
\end{lemma}
\begin{proof} In fact, it is enough to consider the linear
transformation $w=Az$ where $A=(a_{ij})\in GL(n,\C)$ with $a_{ij}=0$
for $i=1,\ldots,d$ and $j=d+1,\ldots,n$. In this way, $B=(b_{ij}) =
A^{-1}$ has the same properties, i.e.; the subspace $W$ given by
$\{z_1=\ldots=z_d=0\}$ is an invariant set under this
transformation. Adjusting the coefficients $a_{ij}$, if necessary,
we can admit that
$$ \dot w_i=\left\{\begin{array}{ll}
    \dps \sum_{j=1}^{d}a_{ij}\dot z_j=\dps\sum_{j=1}^{d}a_{ij}P_j(Bw)=Q_i(w),& \mathrm{ for }\quad
    i=1,\ldots,d\\
    \dps \sum_{j=1}^{n}a_{ij}\dot z_j=\dps\sum_{j=1}^{n}a_{ij}P_j(Bw)=Q_i(w),& \mathrm{ for }\quad
    i=d+1,\ldots,n
    \end{array}\right.$$
having the required properties.
\end{proof}
\par

\subsection{Blowup of subvarieties}
We start by recalling the blowup procedures of $\pn$ along $W$ and
describe the behavior of $\fol$ under this transformation. See
\cite{GH} for more details. Let $\Delta$ be an $n$-dimensional
polydisc with holomorphic coordinates $z_1,\ldots,z_n$ and
$\Gamma\subset \Delta$ the locus $z_1=\ldots=z_{d}=0$. Take
$[y_1:\ldots:y_{d}]$ to be homogeneous coordinates on $\P^{d-1}$.
The blowup of $\Delta$ along $\Gamma$ is the smooth variety
$$
\DC = \lbrace (z,[y])\in \Delta \times \P^{d-1}\, |\, z_i
y_j=z_jy_i\ {\rm for}\ 1\le i,j \le d\rbrace.
$$
The projection $\pi : \DC \rightarrow \Delta$ on the first factor is
an isomorphism away from $\Gamma$, while the inverse image of a
point $z\in \Gamma$ is at projective space $\P^{d-1}.$ The inverse
image $E=\pi^{-1}(\Gamma)$ is the exceptional divisor of the blowup.

The standard open cover $U_j = \lbrace [y_1:\ldots:y_{d}] \,|\,
y_j\ne 0\rbrace$, with $1\leq j\leq d$, of $\P^{d-1}$ yields a cover
of $\widetilde{\Delta}$ where each open set, for $1\leq j\leq d$, is
defined by
\begin{equation}
\label{for100} {\widetilde{U}_j}= \lbrace(z,[y]) \in \DC\ \, |\ \,
[y]\in U_j\rbrace
\end{equation}
with holomorphic coordinates $\sigma_j(u_1,\ldots,u_n)=(z_1,\ldots,z_n)$ given by
$$z_i=
\begin{cases}
u_i & {\rm if}\ i=j\ \mbox{or}\ i=d+1,\ldots,n \\
u_i u_j & {\rm if}\ i=1,\ldots,{\widehat{j}},\ldots,d.
\end{cases}
$$
The coordinates $u\in \C^n$ are affine coordinates on each fiber
$\pi^{-1}(p)\cong\P^{d-1}$ of $E$.

Now consider the algebraic subvariety $W\subset \pn$. Let $\lbrace
\phi_{\lambda},U_{\lambda}\rbrace$ be a collection of local charts
covering $W$ and $\phi_{\lambda}: U_{\lambda}\to \Delta_{\lambda},$
where $\Delta_{\lambda}$ is an $n$-dimensional polydisc. One may
suppose that $\Gamma_{\lambda}=\phi_{\lambda}(W)\cap U_{\lambda}$ is
given by $z_1=\ldots=z_{d}=0.$ Let $\pi_{\lambda}: \DC_{\lambda}\to
\Delta_{\lambda}$ be the blowup of $\Delta_{\lambda}$ along
$\Gamma_{\lambda}$. One can patch the $\pi_{\lambda}$ together and
use the chart maps $\phi_{\lambda}$ to get a blowup
$\widetilde{\P}_n$ of $\pn$ along $W$ and a blowup morphism $\pi :
\widetilde{\P}_n\to \pn$. The exceptional divisor $E$ is a fibre
bundle over $W$ with fiber $\P^{d-1}$, which is naturally identified
with the projectivization $\P(N_{W/_{\pn}})$ of the normal bundle
$N_{W/_{\pn}}$.
\subsection{Nondicritical and Dicritical singularities}
Let us denote by $\widetilde{\fol}$ the strict transform of $\fol$
under $\pi$, that is, $\widetilde{\fol}$ is a foliation by curves on
$\widetilde{\pn}$ whose leaves are projected via $\pi$ in to leaves
of $\fol$. We say that $W$ is \textit{nondicritical} if the
exceptional divisor $E$ is invariant by $\widetilde{\fol}$,
otherwise $W$ is \textit{dicritical}.

Now, by Lemma \ref{lemafp}, we can assume that the foliation $\fol$
is described by vector field $ \ddd_{\fol}
=\dps\sum^{n}_{j=1}P_{j}(z)\frac{\partial}{\partial{z_{j}}}$ such
that
$$q_j=m_{W}(P_j)=
\begin{cases}
m_1 & {\rm for}\ j=1,\ldots,d \\
m_n & {\rm for}\ j=d+1,\ldots,n
\end{cases}
$$
with $m_1\ge m_n$. In order to simplify our notation, we only
consider on the chart $\widetilde{U}_1$ with the coordinate system
given by
$\sigma_1(u)=(u_1,u_1u_2,\ldots,u_1u_d,u_{d+1},\ldots,u_n)=z\in
\C^n$. Thus, for $j\in \{1,d+1,\ldots,n\}$, we have that

\begin{align*}
\dot u_j &= P_j(\sigma_1(u))=\dps \sum_{|a|=q_j} u_1^{a_1}(u_1u_2)^{a_2}\cdots(u_1u_{d})^{a_{d}}P_{j,a}(\sigma_1(u)) \\
 & = u_1^{m_j}\dps\sum_{|a|=q_j}u_2^{a_2}\cdots u_{d}^{a_{d}}
 P_{j,a}(\sigma_1(u)).
\end{align*}
But, $$P_{j,a}(\sigma_1(u))=P_{j,a}(0,\ldots,0,u_{d+1},\ldots,u_n) +
u_1{\widetilde P}_{j,a}(u)= p_{j,a}(u_{d+1},\ldots,u_n)+ u_1
\widetilde{P}_{j,a}(u)$$ and hence
\begin{equation}
\label{equ1en} \dot u_j = u_1^{q_j} \left( \dps
\sum_{|a|=q_j}u_2^{a_2}\cdots u_{d}^{a_{d}}
p_{j,a}(u_{d+1},\ldots,u_n) + u_1
\widetilde{P_j}(u)\right):=u_1^{q_j}\left(Q_j(u)+u_1\widetilde{P_j}(u)\right)
\end{equation}
for some functions $\widetilde{P_j}$.

Finally, for $j\in\{2,\ldots,d\}$, since $z_j=u_1u_j$ we obtain that

\begin{equation}
\label{equ3en} \dot u_j = P_j(\sigma_1(u))=u_1^{m_1-1} \left( G_j(u)+
u_1 \widetilde{P_j}(u)\right)
\end{equation}
where
\begin{equation}
\label{equ4en} G_j(u):=\dps \sum_{|a|=m_1}u_2^{a_2}\cdots
u_{d}^{a_{d}} \big(p_{j,a}(u_{d+1},\ldots,u_n) -u_j\cdot
p_{1,a}(u_{d+1},\ldots,u_n)\big):=Q_j(u)-u_jQ_1(u)
\end{equation}
for some functions $\widetilde{P_j}$. With the equations
(\ref{equ1en}), (\ref{equ3en}) and (\ref{equ4en}), we obtain the
expressions of $\pi^*\fol$ (in a similar way with the others charts
$\widetilde{U}_j$),

\begin{align}
\label{equ5en} \ddd_{\pi^*\fol}=&\dps
u_1^{m_1}[Q_1(u)+u_1\widetilde{P_1}(u)]\frac{\partial}{\partial u_1}
+
u_1^{m_1-1}\sum_{j=2}^{d}[G_j(u)+u_1\widetilde{P_j}(u)]\frac{\partial}{\partial
u_j}+\cr
&u_1^{m_n}\sum_{j=d+1}^{n}[Q_j(u)+u_1\widetilde{P_j}(u)]\frac{\partial}{\partial
u_j}.
\end{align}
The vector field that generates $\widetilde{\fol}$ is obtained from
(\ref{equ5en}), after a division by an adequate power of $u_1$. The
situation where $W$ is a nondicritical component is described in the
following cases

\indent (i) $m_n + 1 = m_1$ and $G_{j_0}\not\equiv 0$ for some
$j_0\in\{2,\ldots,d\}$.

\noindent In this condition, dividing (\ref{equ5en}) by $u_1^{m_n}$
we get
\begin{align}
\label{equ6en} \ddd_{\widetilde{\fol}}=&\dps
u_1\big(Q_1(u)+u_1\widetilde{P_1}(u)\big)\frac{\partial}{\partial
u_1} +
\sum_{j=2}^{d}\big(G_j(u)+u_1\widetilde{P_j}(u)\big)\frac{\partial}{\partial
u_j}+\cr
&\sum_{j=d+1}^{n}\big(Q_j(u)+u_1\widetilde{P_j}(u)\big)\frac{\partial}{\partial
u_j}.
\end{align}\\
If there is some $G_{j_i}\equiv0$ then the set
$\sing(\widetilde{\fol})\cap E$ contains some component of positive
dimension. Furthermore, in order to $\fol$ be special along $W$ then
necessary we have that $G_j\not\equiv0$, $ \forall j$.

\indent(ii) $m_n + 1 < m_1$.

\noindent As in the case before, dividing (\ref{equ5en}) by
$u_1^{m_n}$ we get
\begin{align}
\label{equ7en} \ddd_{\widetilde{\fol}}=&\dps
u_1^{m_1-m_n}\big(Q_1(u)+u_1\widetilde{P_1}(u)\big)\frac{\partial}{\partial
u_1} +
u_1^{m_1-m_n-1}\sum_{j=2}^{d}\big(G_j(u)+u_1\widetilde{P_j}(u)\big)\frac{\partial}{\partial
u_j}+\cr
&\sum_{j=d+1}^{n}\big(Q_j(u)+u_1\widetilde{P_j}(u)\big)\frac{\partial}{\partial
u_j}.
\end{align}
In this situation, the leaves of $\widetilde{\fol}$ when restricted
on the exceptional divisor $E$ are contained in the hyperplane
$u_j=k_j$, for $j\in\{2,\ldots,d\}$. Furthermore, always there exist
non-isolated singularities of $\widetilde{\fol}$ on $E$.

\indent(iii) $m_n = m_1$ and $G_{j_0}\not\equiv 0$ for some
$j_0\in\{2,\ldots,d\}$.

\noindent We may divide (\ref{equ5en}) by $u_1^{m_1-1}$. As
consequence, we obtain
\begin{align}
\label{equ8en} \ddd_{\widetilde{\fol}}=&\dps
u_1\big(Q_1(u)+u_1\widetilde{P_1}(u)\big)\frac{\partial}{\partial
u_1} +
\sum_{j=2}^{d}\big(G_j(u)+u_1\widetilde{P_j}(u)\big)\frac{\partial}{\partial
u_j}+\cr
&u_1^{m_n-m_1+1}\sum_{j=d+1}^{n}\big(Q_j(u)+u_1\widetilde{P_j}(u)\big)\frac{\partial}{\partial
u_j}.
\end{align}
Now, the leaves of $\widetilde{\fol}$ on $E$ are contained in the
hyperplane $u_j=k_j$ for $j\in\{d+1,\ldots,n\}$.
\par On the other hand, when $W$ is dicritical, we have the following
condition:\\
\noindent(i) $m_n = m_1$ and $G_j\equiv0$ for $j\in\{2,\ldots,d\}$.

\noindent Therefore, after the division of (\ref{equ5en}) by
$u_1^{m_1}$ we get
\begin{align}
\label{equ9en} \ddd_{\pi^*\fol}=\dps
\big(Q_1(u)+u_1\widetilde{P_1}(u)\big)\frac{\partial}{\partial u_1}
+ \sum_{j=2}^{d}\big(\widetilde{P_j}(u)\big)\frac{\partial}{\partial
u_j}+
\sum_{j=d+1}^{n}\big(Q_j(u)+u_1\widetilde{P_j}(u)\big)\frac{\partial}{\partial
u_j}.
\end{align}
Note that the exceptional divisor $E$ is not an invariant set by
$\widetilde{\fol}$.

\begin{lemma}\label{lemlll}
With the notation given in the Lemma \ref{lemafp}, we get
\begin{enumerate}
\item $m_{W}(\fol)={\rm min}\{m_1,m_n\}$;
\item Let $\ll_{\fol}$ and $\ll_{\widetilde{\fol}}$ be tangent sheaves of the foliations $\fol$ and $\widetilde{\fol}$, respectively.  If $\ell$ is the integer such that
$$
\ll_{\widetilde{\fol}} \cong \pi^*\ll_{\fol} \otimes \oo_{\widetilde{Y}}(\ell E)
$$
then
\begin{equation*}
\ell=
\begin{cases}
{\rm min}\lbrace m_1-1, m_n \rbrace & \mbox{ if } {W \mbox{ is nondicritical}} \\
m_1& \mbox{ if } {W \mbox{ is dicritical. }}
\end{cases}
\end{equation*}
In particular, if $m_1=m_n$ then
\begin{equation*}
\ell=
\begin{cases}
m_{W}(\fol)-1& \mbox{ if } {W \mbox{ is nondicritical}} \\
m_{W}(\fol)& \mbox{ if } {W \mbox{ is dicritical. }}
\end{cases}
\end{equation*}

The number $\ell:=m_{E}(\folt,W)$ will be called the order of
vanishing of $\folt$ at $E$.
\end{enumerate}

\end{lemma}
\subsection{Chern Class}
Let $\pi:\pnt\to \pn$, $n \ge 3$, be the blowup of $\P^n$ centered
at $W$, with exceptional divisor denoted by $E$. Set
$N:=N_{W/\P^n}$ and $\rho:=\pi|_{E}$. Since $E \cong \P(N)$,
recall that $A(E)$ is generated as an $A(\calc)$-algebra by the
Chern class
$$
\zeta := c_1(\oo_{N}(-1))
$$
with the single relation
\begin{equation}
\label{equdch} \zeta^{d} -
\rho^*c_1(N)\zeta^{d-1}+\ldots+(-1)^{d-1}\rho^*c_{d-1}(N)\zeta +
(-1)^{d}\rho^*c_{d}(N)=0.
\end{equation}
The normal bundle $N_{E/\pnt}$ agrees with the tautological bundle
$\oo_{N}(-1)$, and hence
\begin{equation}
\label{equzet} \zeta=c_1(N_{E/\pnt}).
\end{equation}
If $\iota :E\hookrightarrow \pnt$ is the inclusion map, we also get
\begin{equation}
\label{equslf} \iota_{*}(\zeta^i)=(-1)^i E^{i+1}.
\end{equation}

Now, from Porteous Theorem (cf. \cite{IP}, \cite{WF}), writing
$c_j(\pnt)$ and $c_j(\pn)$ for $c_j(T_{\pnt})$ and $c_j(T_{\pn})$,
respectively, we have
\begin{equation}
\label{for20} c(\widetilde{\pn}) - \pi^*c(\pn) = \iota_{*}(\rho^*C(T_{W})\cdot \alpha)
\end{equation}
where
$$\alpha = \frac{1}{\zeta}\dps\sum_{i=0}^{d}\big[1-(1-\zeta)(1+\zeta)^i\big]\rho^*c_{d-i}(N).$$
In this way, we can rewrite $\alpha$,
\begin{equation}
\label{alpha1}
\alpha = \dps\sum_{m=0}^{d}\sum_{j=0}^{d-m}\bigg[{{d-m}\choose{j}}-{{d-m}\choose{j+1}}\bigg]\zeta^j\rho^*c_m(N)=\dps\sum_{i=0}^{d}\alpha_i
\end{equation}
where
\begin{equation}
\label{alpha2}
\alpha_i=\sum_{m=0}^{i}\bigg[{{d-m}\choose{i-m}}-{{d-m}\choose{i-m+1}}\bigg]\zeta^{i-m}\rho^*c_m(N).
\end{equation}
Writing $(\rho^*C(W)\cdot \alpha)=\dps\sum_{j=0}^{n}\beta_j$ where
$$\beta_j=\sum_{i=0}^{j}\alpha_i \rho^*c_{j-i}(T_{W}),$$
we get,
\begin{equation}
\label{fcc}
c_j(\pnt)=\pi^*c_j(\pn)+\iota_*(\beta_{j-1}).
\end{equation}
Now, defining $a=(a_1,a_2)\in \Z^2$, with $0\le a_1:=m\le d$, $0\le a_2=j-1-i\le n-d$ and $|a|=a_1+a_2$, we can rewrite (\ref{fcc}) in the following way

\begin{equation}
\label{fcc1}
c_j(\widetilde{\pn}) =\pi^*c_j(\pn)+\dps\sum_{|a|=0}^{j-1}(-1)^{j-|a|-1}\Gamma^{j}_{a}\cdot\rho^*c_{a_1}(N)\rho^*c_{a_2}(T_{W})E^{j-|a|}
\end{equation}
where
\begin{equation}
\label{gamma}
\Gamma^{j}_{a}={{d-a_1}\choose{j-|a|-1}}-{{d-a_1}\choose{j-|a|}}.
\end{equation}
Here we are assuming that
$\dps{{p}\choose{q}}=0$ if $p<q$ or $q<0$.

Now, let $i:W \to \pn$ be a closed embedding of $W$ in $\pn$. From the exact sequence
$$ 0 \to T_{W} \to i^*T_{\pn} \to N \to 0$$
one has
$$c(T_{W}) = c(i^*T_{\pn})/c(N).$$
If $W$ is an intersection of divisors $D_1,\ldots,D_d$, then
$$c(N)=\dps\prod_{j=1}^{d}(1+c_1(i^*\mathcal{O}_{\pn}(D_i)))=\prod_{j=1}^{d}(1+k_j\mathbf{h})=\sum c_i(N)$$
where $\mathbf{h}$ is the class of a hyperplane of $\pn$ and $k_j$
is the $deg(D_j)$. Therefore,
\begin{equation}
\label{ccni} c_i(N)=\sigma_{i}^{(d)}\cdot {\mathbf h}^i.
\end{equation}
From the projection formula one deduces that
$$c(T_{W})= \dps\frac{(1+\mathbf{h})^{n+1}}{\dps\prod_{j=1}^{d}(1+k_j\mathbf{h})}=\sum c_i(T_W).$$
Consequently,
\begin{equation}
\label{ccti} c_i(T_{W}) =\tau_i^{(d)}\cdot \mathbf{h}^i.
\end{equation}

Repeatedly applying the equation (\ref{equdch}), we get
\begin{equation}
\label{zetger} \zeta^{m}=\dps\sum_{i=1}^{d}\rho^*(\beta_{m,i}\cdot
\mathbf{h}^{m-d+i})\zeta^{d-i}
\end{equation}
where
$$\beta_{m,i}=\dps\sum_{j=0}^{i-1}(-1)^{j}\sigma_{j}^{(d)}\cdot\mathcal{W}_{m-d+i-j}^{(d)}$$
for $m\ge d$. In fact, it is clear for $m=d$ because
$$\dps\sum_{j=0}^{m}(-1)^j\sigma_j^{(d)}\cdot\mathcal{W}_{m-j}^{(d)}=0,\quad m>0$$
and $\beta_{d,i}=(-1)^{i-1}\sigma_{i}^{(d)}$. Assuming that the
induction hypothesis is true for $m\ge d$, we get
\begin{align*}
\zeta^{m+1}&=\zeta\cdot\zeta^{m}=\dps\sum_{i=1}^{d}\rho^*(\beta_{m,i}\cdot
\mathbf{h}^{m-d+i})\zeta^{d-i+1}\\
&=\rho^*(\beta_{m,1}\cdot
\mathbf{h}^{m-d+1})\zeta^{d}+\dps\sum_{i=1}^{d-1}\rho^*(\beta_{m,i+1}\cdot
\mathbf{h}^{m-d+i+1})\zeta^{d-i}.\\
\end{align*}
Again, from (\ref{equdch}), we yields
\begin{align*}
\beta_{m+1,i}&=(-1)^{i-1}\sigma_{i}^{(d)}\cdot
\mathcal{W}_{m-d+1}^{(d)}+\beta_{m,i+1}\\
&=\dps\sum_{j=1}^{i-1}(-1)^{j}\sigma_{j}^{(d)}\cdot\mathcal{W}_{m+1-d+i-j}^{(d)}.
\end{align*}
In particular, $\beta_{m,1}=\mathcal{W}_{m-d+1}^{(d)}$ which results
\begin{equation}
\label{cce}
\dps\int_{E}\rho^*c_{i}(N)\zeta^{n-i-1}=\left\{\begin{array}{ll}
                                                           (-1)^{d-1}\sigma_{i}^{(d)}\mathcal{W}_{n-d-i}^{(d)}\deg(W),&0\le i\le n-d\cr
                                                           0,&i>n-d.
                                                   \end{array}\right.
\end{equation}
\subsection{Embedded closed points}
\par Let $\fol ol(\pn;k)$ be the space of holomorphic foliations by curves of degree $k\geq 1$ of $\pn$.
Denote by $N(\pn;k)\subset\fol ol(\pn;k)$ the set of non-degenerate
foliations of degree $k$. Since the Jouanolou's foliation given in
(\ref{jouanolou}) is non-degenerated, we have that $N(\pn;k)$ is not
empty. The following result guarantees that a foliation by curves
with singularities of positive dimension can be deformed to a
non-degenerated foliation of $\pn$.
\begin{theorem}[Soares \cite{bound}]\label{marcio-deformation}
 $N(\pn;k)$ is open, dense and connected in $\fol ol(\pn;k)$.
 \end{theorem}
\par We note that this result can not be used to determine the \textit{Milnor numbers}  of an irreducible component of $\sing(\fol)$,
because the deformation is not explicit. In order to prove Theorem
\ref{theorem1}, we use the notion of \textit{embedded closed points
associated to $W$}. In fact, we will see that it is possible to
deform $\fol$ to a foliation by curves $\fol_t$ on $\pn$ (for some
small $0<|t|<\epsilon$) such that, either $W$ is $\fol_t$-invariant
or else $\fol_t$ is special along $W$ (see for instance Lemma
\ref{lema_Theorem1}). Therefore, these isolated singularities that
appear after this deformation  and do not lie in $W$ will be called
\textit{embedded closed points associated to $W$}. More precisely,
the set of embedded closed points of $W$ is
\begin{equation}
A_{W}:=\{p^{t}_{j}\in\sing(\fol_t): \displaystyle\lim_{t\to
0}p^{t}_{j}\in W,\quad\text{such that}\quad p^{t}_{j}\not\in W\quad
\forall t\ne 0\} .
\end{equation}
\par Let $X\subset\pn$ be a non-empty subset. Here and subsequently $N(\fol,X)$ denotes the number of isolated singularities of $\fol$ on $X$ (counted with multiplicity). Hence, $$N(\fol,A_{W})=\sum_{q^{t}_{j}\in A_{W}}\mu(\fol_t,q^{t}_{j}).$$
\par  Now, suppose that $V\subset\pn$ is a smooth irreducible complete intersection defined by $$V=Z(F_{1},\ldots,F_{n-m})$$ where $F_{l}\in\C[z_{0},\ldots,z_{n}]$ is homogeneous of degree $d_l$, $1\leq l\leq n-m$. In the proof of Theorem \ref{principal}, we will use the following result due to Soares in
\cite{soares}.
\begin{proposition}\label{result}
Assume that $V$ is invariant by $\fol$, where $\fol$ is a foliation
by curves of $\pn$ of degree $k$, non-degenerate along $V$. Then the
number of singularities of $\fol$ in $V$ is
$$N(\fol,V)=(d_1\ldots d_{n-m})\sum^{m}_{j=0}\left [ \sum^{j}_{\delta=0}(-1)^\delta\mathcal{W}^{(n-m)}_{\delta}(d_1-1,\ldots,d_{n-m}-1) \right ]k^{m-j},$$ where $\mathcal{W}^{(n-m)}_{\delta}$ is the Wronski of degree $\delta$ and $n-m$ variables.
\end{proposition}


\section{Deformation of a foliation by curves adapted to singular set}
We use the following fundamental lemma.
\begin{lemma}\label{lema_Theorem1}
Let $\fol$ be a holomorphic foliation by curves on $\P^n$, $n\geq
3$, of degree $k$. Suppose that
$$
 \mbox{Sing}(\fol) = W \cup\{p_1,\ldots,p_s\}
$$
is a disjoint union where $W=Z(h_{1},\ldots,h_{d})$ is a smooth
complete intersection of $\pn$ and $p_j$ are isolated points. Let
$\ell=m_E(\fol,W)$. Suppose that
$$V=Z(f_1,\ldots,f_m)$$ is $\fol$-invariant such that $f_i\in
\mathcal{I}(W)$ for all $1\leq i\leq m$. Then, $W\subset V$ and
there exists a one-parameter family of holomorphic foliations by
curves on $\pn$, given by $\{\fol_t\}_{t\in D}$ where
$D=\{t\in\C\,:\,|t|<\epsilon\}$ satisfying:
\begin{enumerate}
\item[(i)] $\fol_0=\fol$;
\item[(ii)] $\deg(\fol_t)=\deg(\fol)$;
\item[(iii)] If $\ell=0$ then $\sing(\fol_t)=\{p_1^t,\ldots,p_{s_t}^t\}$ , if $\ell>0$ then $\sing(\fol_t)=W
\cup\{p_1^t,\ldots,p_{s_t}^t\}$, disjoint union where $p_i^t$ are
isolated points, for any $t\in D\setminus\{0\}$;
\item[(iv)] If $\ell=0$ then $W$ is $\fol_t$-invariant, for any $t\in D\setminus\{0\}$;
\item[(v)] If $\ell>0$ then $W\subset\sing(\fol_t)$, $\ell=m_E(\fol_t,W)$ and $\fol_t$ is special along $W$ for any $t\in D\setminus\{0\}$;
\item[(vi)] $V$ is $\fol_t$-invariant, for any $t\in D$.
\end{enumerate}
\end{lemma}

\begin{proof}
For simplicity we assume that $W$ is given in the affine coordinates
$(z_1,\ldots,z_n)$ by the zeros of the polynomials
$h_{1},\ldots,h_{d}$ and  $V$ is given by the zeros of the
polynomials $f_1,\ldots,f_{m}$, with $f_i\in \mathcal{I}(W)$. Now
let us consider the holomorphic function
\begin{gather*}
\begin{matrix}
F\  : &  \C^n\  & \longrightarrow & \C^{n}   \\
      & z=(z_1,\ldots,z_n)  & \longmapsto     & \bigg(h_1(z)\, ,\,\ldots\, ,\, h_{d}(z)\,,z_{d+1},\ldots,\, z_n\bigg).
\end{matrix}
\end{gather*}
By hypotheses we get
$$\begin{array}{cccc}
    f_1=&\lambda_{11}h_1 &+\ldots+&\lambda_{1d}h_{d} \\
   \vdots & \vdots & \vdots & \vdots\\
   f_{m}= & \lambda_{m1}h_1 & + \ldots+  &\lambda_{md}h_{d},
\end{array}$$
where $\lambda_{ij}\in\C$ for all $1\leq i\leq m$ and $1\leq j\leq
d$.
    Without loss of generality, since $W$ is smooth, we can admit
that $Det(M(z))\ne0$ where $M(z)=\left(\frac{\partial h_i}{\partial
z_j}\right)$ is the $d\times d$-minor of the Jacobian matrix $D_z F$
for $z$ in some open set $U \subset \C^n$. Thus, note that $F|_{U}$
is a local biholomorphism onto an open set
$\widetilde{U}\subset\C^n$ and restricted to $\widetilde{U}$ so one
may fix coordinates $w=F(z)$.

Then the image of $W\cap U$ by $F$ is the hyperplane of codimension
$d$,
$$H_{d}:=\{(w_1,\ldots,w_{n})\in\C^n:w_1=\ldots=w_{d}=0\}$$
while the image of $V\cap U$ by $F$ is given by
$$H_{m}:=\{(w_1,\ldots,w_{n})\in\C^n: \dps\sum_{j=1}^{d}\lambda_{ij}w_j=0,\quad 1\le i\le m\}.$$
To continue we describe the push-forward $(F)_*\fol$. We write
$$
\ddd_{(F)_*\fol}=P_1(w)\,\frac{\partial}{\partial w_1}+\ldots
+P_n(w)\,\frac{\partial}{\partial w_n}
$$
where
\begin{equation}
\label{equgil}
P_i(w)=\sum_{|a|=m_i}w_1^{a_1}\cdots w_{d}^{a_{d}}P_{i,a}(w)
\end{equation}
with at least one $P_{i,a}(z)$ not vanishing in $H_{d}$. Therefore we define $\G_{t}$ by the vector field
\begin{equation}\label{deformation}
\ddd_{\G_{t}}=\ddd_{(F)_*\fol}+t\cdot Y
\end{equation}
where
$$Y=Q_1\,\frac{\partial}{\partial w_1}+\ldots +Q_n\,\frac{\partial}{\partial w_n}
$$
with
$$Q_j(w)=\begin{cases}
   \dps\sum_{|I|=q_j}a_{I,j}w_{1}^{i_1}\cdots w_{d}^{i_d} & \text{if}\qquad 1\leq j\leq d\\
   \dps\sum_{|I|=q_j}R_{I,j}(w)w_{1}^{i_1}\cdots w_{d}^{i_d}  & \text{if}\qquad d+1\leq j\leq n.
\end{cases}$$
where  $a_{I,j}\in\C$, $R_{I,j}$ are affine linear functions and
\begin{equation}
\label{rem} q_1=\ldots=q_{d}=q_{d+1}+1=\ldots=q_n+1=\ell+1.
\end{equation}

\par One can choose $a_{I,j}$ in way
that $H_m$ are invariants by $Y$( rank-nullity theorem). If $\ell=0$
then follows from (\ref{deformation}) that $H_d$ and $H_m$ are
invariant by $\G_t$ for any $t\in D\setminus\{0\}$, indeed
$$\ddd_{\G_{t}}|_{H_{d}}=\ddd_{(F)_*\fol}|_{H_{d}}+t\cdot
Y|_{H_{d}}=t\cdot
Y|_{H_{d}}=t\sum_{j=d+1}^{n}R_{0,j}(w)\frac{\partial}{\partial
w_j},$$ because $q_{d+1}=\ldots=q_{n}=0$.

\par Now, if $\ell >0$ then $W\subset\sing(\G_t)$ and $\ell=m_{E}(\G_t,W)$ for all $t\in D\setminus\{0\}$. Again, $H_m$ is also invariant by
$\G_t$.

\par Denoting by $\fol_t$ the pull-back foliation of $\G_t$ by $F$, we
have that $\fol_t$ define a one-dimensional singular holomorphic
foliation in $U$. Indeed,  this foliation is given by a vector field
$$\ddd_{\fol_{t}}=P^{t}_{1}\frac{\partial}{\partial z_1}+\ldots +P^{t}_{n}\,\frac{\partial}{\partial z_n}$$
where $P^{t}(z)=(P^{t}_1(z),\ldots,P^{t}_n(z))$ are obtained by the
system $$Q^{t}\circ F(z)=D_zF(z)\cdot P^{t}(z),$$ with
$Q^{t}=(P_1+tQ_1,\ldots,P_n+tQ_n)$. Solving this system by Cramer's
rule, we get
$$P^{t}_{i}(z)=\frac{\det(A^{t}_i)}{\det(M)}$$
where one gets $A^{t}_i$ replacing the $i$-th column of $DF$ by the
vector column $Q^{t}\circ F(z)$. In particular,
$$P^{t}_{i}(z)=\frac{Q^{t}_{i}\circ F(z)\cdot\det(M)}{\det(M)}=Q^{t}_{i}\circ F(z)\cdot\det(M)\qquad\text{for all}\qquad1\leq i\leq d.$$
Therefore, after normalizing by the factor $\det(M)$, one may write
$\fol_t$ in $U$ as
\begin{eqnarray*}
\ddd_{\fol_{t}} & = & \det(A^{t}_1)\frac{\partial}{\partial{z_1}}+\ldots+\det(A^{t}_d)\frac{\partial}{\partial{z_d}} + \\
 &  & + Q^{t}_{d+1}\circ F(z)\cdot\det(M)\frac{\partial}{\partial{z_{d+1}}}+\ldots+Q^{t}_{n}\circ F(z)\cdot\det(M)\frac{\partial}{\partial{z_{n}}}.
\end{eqnarray*}
Since the components of $\ddd_{\fol_{t}}$ are polynomials, we can
consider $\fol_t$ defined in $\C^n$ by Hartogs extension Theorem.
Moreover, $V$ is invariant by $\fol_{t}$ (for any $t\in
D\setminus\{0\}$), hence it proved $(iv)$.

By construction, $F$ is a local biholomorphism and adjusting the
functions $R_{I,j}$ and the constant $a_{I,j}$ and shrinking
$\epsilon$, if necessary, we can admit that $\fol_t$ is special
along $W$ and
$\sing\big(\fol_t\big)=W\cup\{p_1^t,\ldots,p_{s_t}^t\}$, if $\ell>0$
and $\sing\big(\fol_t\big)=\{p_1^t,\ldots,p_{s_t}^t\}$, if $\ell=0$,
where each $p_i^t$ is an isolated point, proving $(iii)-(v)$. Also
by construction, $\fol_t$ satisfies the (vi)-condition. It is
immediate that $\fol_0=\fol$ proving $(i)$. The coefficients of
$R_{I,j}$ are chosen so that we have $\deg(\fol_t)=\deg(\fol)$
proving $(ii)$.
\end{proof}
It follows from Lemma \ref{lema_Theorem1} that we need to describe
two situations: the case in that $\fol$ is special along $W$ and the
case in that $W$ is invariant by $\fol$. First we assume that $\fol$
is special along $W$. In particular, the following result gives a
proof of Theorem \ref{second}.
\begin{theorem}\label{nunsing1}
Let $\fol$ be a holomorphic foliation by curves of degree $k\geq 1$
defined in $\P^n$ such that
$$\sing{\fol}=W\cup\{p_1,\ldots,p_{s}\}$$ is a disjoint union where $p_i$ are isolated singularities. Assume $\fol$ special along $W=Z(h_1,\ldots,h_d)$ with
$k_j=\mathrm{deg}(h_j)$. Then
$$
\sum_{i=1}^{s}\mu(\fol,p_i)=\dps\sum_{i=1}^{n}k^i +\nu(\fol,W),
$$
where $\nu(\fol,W)$ is given as in (\ref{nu_equation}).
\end{theorem}
\begin{proof}
Let $\pi:\widetilde{\P}^n\to\pn$ be the blowup of $\pn$ along $W$,
with exceptional divisor $E$. Given that $\fol$ is special along
$W$, the singular set of $\widetilde\fol$ consists only of closed
points. Then

$$\sum_{j=1}^{s}\mu(\fol,p_j)=N(\widetilde{\fol},\pnt)-N(\widetilde{\fol},E),$$
because $\pi$ is an isomorphism away from $W$. By \cite{BB}, we have that
$$N(\widetilde\fol,\pnt) = \dps\int_{\pnt}c_n\big(T_{\pnt}\otimes\ll_{\widetilde{\fol}}^*\big)$$
where
\begin{equation}
\label{ff1} c_n(T_{\pnt}\otimes\ll_{\widetilde{\fol}}^*) =
\dps\sum_{j=0}^{n}c_j(T_{\pnt})\cdot
c_1^{n-j}(\ll_{\widetilde{\fol}}^*)
\end{equation} and by Lemma
\ref{lemlll}
\begin{equation}
\label{ccftg}
c_1^{n-j}(\ll_{\widetilde{\fol}}^*)=\dps\sum_{m=0}^{n-j}{{n-j}\choose{m}}\pi^*c_1^{m}(\ll_{\fol}^*)(-\ell
E)^{n-j-m}.
\end{equation}
In order to simplify our notation, each term of the expression
$\dps\int_{\pnt} c_j(T_{\pnt})\cdot
c_1^{n-j}(\ll_{\widetilde{\fol}}^*)$ will be computed as the sum of
two terms separately. Thus, from (\ref{fcc1}) its first term is
$$\dps\int_{\pnt}\pi^*c_j(T_{\pn})c_1^{n-j}(\ll_{\widetilde{\fol}}^*)=\sum_{m=0}^{n-j}\int_{\pnt}{{n-j}\choose{m}}\pi^*c_j(T_{\pn})\pi^*c_1^{m}(\ll_{\fol}^*)(-\ell
E)^{n-j-m}.$$ Therefore, in this last expression, it is enough to
focus only for $m=n-j$ and $m+j\le n-d$. Consequently,
\begin{align*}
\dps\int_{\pnt}\pi^*c_j(T_{\pn})c_1^{m}(\ll_{\widetilde{\fol}}^*)&=\int_{\pnt}\pi^*c_{j}(\pn)\pi^*c_1^{n-j}(\ll_{\fol}^*)+\\
&\dps\sum_{m=0}^{n-d-j}\int_{\pnt}{{n-j}\choose{m}}\rho^*c_j(\pn)\rho^*c_1^{m}(\ll_{\fol}^*)(-\ell E)^{n-j-m}.
\end{align*}
Therefore, the first term results in

\begin{equation}
\label{pripar}
\begin{array}{ll}
\dps\int_{\pnt}&\pi^*c_j(T_{\pn})\cdot c_1^{n-j}(\ll_{\widetilde{\fol}}^*)=\dps{{n+1}\choose{j}}(k-1)^{n-j}+\\
&\dps\sum_{m=0}^{n-d-j}\int_{E}{{n-j}\choose{m}}\rho^*c_j(\pn)\rho^*c_1^{m}(\ll_{\fol}^*)(-\ell)^{n-j-m}\zeta^{n-j-m-1}.
\end{array}
\end{equation}
In a similar way as computed in (\ref{cce}), we obtain the following
sum
\begin{equation}
\label{pripar1}
\begin{array}{ll}
\dps\sum_{j=0}^{n}\int_{\pnt}&\pi^*c_j(T_{\pn})\cdot c_1^{n-j}(\ll_{\widetilde{\fol}}^*)=1+k+k^2+k^3+\ldots+k^{n-1}+k^n +\\
&\dps\sum_{j=0}^{n-d}\sum_{m=0}^{n-d-j}(-1)^{\delta_j^m-1}{{n+1}\choose{j}}{{n-j}\choose{m}}(k-1)^m\ell^{n-j-m}\mathcal{W}_{\delta_j^m}^{(d)}\deg(W)
\end{array}
\end{equation}
where $\delta_{j}^{m}=n-d-j-m$. Now, the second term of (\ref{ff1})
is

$$\dps\int_{\pnt}\sum_{|a|=0}^{j-1}(-1)^{j-|a|-1}\Gamma_{a}^j\rho^*c_{a_1}(N)\rho^*c_{a_2}(T_{W})c_1^{n-j}(\ll_{\widetilde{\fol}}^*)E^{j-|a|}.$$

In this time, given that $|a|+m\le j-1 + n-j = n-1$, it is enough to
consider $m+|a|\le n-d$. In the same way as done in the case before,
the second term of (\ref{ff1}) results

\begin{equation*}
\dps\sum_{|a|=0}^{j-1}\sum_{m=0}^{n-d-|a|}(-1)^{\delta_{|a|}^{m}}{{n-j}\choose{m}}\Gamma_{a}^{j}\ell^{n-j-m}(k-1)^m\sigma_{a_1}^{(d)}\tau_{a_2}^{(d)}\mathcal{W}_{\delta_{|a|}^{m}}^{(d)}\deg(W).
\end{equation*}
Consequently, after adding these terms,, we get
$$
\begin{array}{rl}
&N(\widetilde{\fol},\pnt)=1+k+k^2+k^3+\ldots+k^{n-1}+k^n+\cr
&\deg(W)\dps\sum_{j=0}^{n-d}\sum_{m=0}^{n-d-j}(-1)^{\delta_j^m-1}{{n+1}\choose{j}}{{n-j}\choose{m}}(k-1)^m\ell^{n-j-m}\mathcal{W}_{\delta_j^m}^{(d)}+\cr
&\deg(W)\dps\sum_{j=1}^{n}\sum_{|a|=0}^{j-1}\sum_{m=0}^{n-d-|a|}(-1)^{\delta_{|a|}^{m}}{{n-j}\choose{m}}\Gamma_{a}^{j}\ell^{n-j-m}(k-1)^m\sigma_{a_1}^{(d)}\tau_{a_2}^{(d)}\mathcal{W}_{\delta_{|a|}^{m}}^{(d)}.
\end{array}
$$
By Whitney sum, we know that $c(T_{\pn})|_{W}=c(N)|_{W}\cdot
c(T_{W})$, i.e.,
$$\dps{{n+1}\choose{j}}=\sum_{a_1+a_2=j}\sigma_{a_1}^{(d)}\tau_{a_2}^{(d)},$$
we get
\begin{equation}
\label{nspnt} \begin{array}{ll}
&N(\widetilde{\fol},\pnt)=1+k+k^2+k^3+\ldots+k^{n-1}+k^n+\cr
&\deg(W)\dps\sum_{|a|=0}^{n-d}\sum_{j=|a|}^{n}\sum_{m=0}^{n-d-|a|}(-1)^{\delta_{|a|}^{m}}{{n-j}\choose{m}}\Gamma_{a}^{j}\ell^{n-j-m}(k-1)^m\sigma_{a_1}^{(d)}\tau_{a_2}^{(d)}\mathcal{W}_{\delta_{|a|}^{m}}^{(d)}.
\end{array}
\end{equation}

Now, we will compute the number of isolated singularities, counted with multiplicity, of $\widetilde{\fol}$ on $E$. Again, by  Baum-Bott's formula, we have that
$$N(\widetilde{\F}, E)= \int_{E}c_{n-1}(\tt_{E} \otimes \ll_{\widetilde{\fol}}^*)$$
with
$$
c_{n-1}(\tt_{E} \otimes \ll_{\widetilde{\fol}}^*) = \sum_{i=0}^{n-1}c_i(E)\cdot c_1(\ll_{\widetilde{\fol}}^*)^{n-i-1}.
$$
On the one hand,
\begin{equation}
\label{equwit}
c_i(E) = c_i(\tt_{\pnt}\otimes\oo_E)-c_{i-1}(E) \zeta
\end{equation}
and reapplying (\ref{equwit}) recursively we obtain
$$
c_i(E)=\dps\sum_{j=0}^{i}(-1)^jc_{i-j}(\tt_{\pnt}\otimes\oo_E)
\zeta^{j}.
$$
Then, using also (\ref{fcc}), we get
\begin{align*}
c_{i}(E)&=\sum_{j=0}^{i}(-1)^{i-j}\pi^*c_{j}(\pn)\zeta^{i-j} +\\
&\ \ \ +\dps\sum_{j=0}^{i}\sum_{|a|=0}^{j-1}(-1)^{i-|a|-1}\Gamma_{a}^{j}\rho^*c_{a_1}(N)c_{a_2}(T_{W})\zeta^{i-|a|}\\
&=\sum_{|a|=0}^{i}(-1)^{i-|a|}\rho^*c_{a_1}(N)c_{a_2}(T_{W})\zeta^{i-|a|} +\\
&\ \ \
+\dps\sum_{|a|=0}^{i-1}\sum_{j=|a|+1}^{i}(-1)^{i-|a|-1}\Gamma_{a}^{j}\rho^*c_{a_1}(N)c_{a_2}(T_{W})\zeta^{i-|a|}
\end{align*}
Given that
$$\sum_{j=|a|+1}^{i}\Gamma_{a}^{j}=\dps1-{{d-a_1}\choose{i-|a|}}$$
follows that
\begin{equation}
\label{ccde}
c_{i}(E)=\dps\sum_{|a|=0}^{i}(-1)^{i-|a|}{{d-a_1}\choose{i-|a|}}\rho^*c_{a_1}(N)c_{a_2}(T_{W})\zeta^{i-|a|}
\end{equation}
On the other hand, as
$c_1(\ll_{\widetilde{\fol}}^*)=\pi^*c_1(\ll_{\fol}^*) -\ell\zeta$,
we have
\begin{equation}
\label{ftf}
c_1(\ll_{\widetilde{\fol}}^*)^{n-i-1}=\dps\sum_{m=0}^{n-i-1}{{n-i-1}\choose{m}}\pi^*c_1(\ll_{{\fol}}^*)^{m}(-\ell\zeta)^{n-i-m-1}.
\end{equation}
Again, as done in the case before, we obtain that
$N(\widetilde{\fol},E)$ is equal to $$
\dps\deg(W)\sum_{i=0}^{n-1}\sum_{|a|=0}^{i}\sum_{m=0}^{n-d-|a|}(-1)^{\delta_{|a|}^{m}}{{n-i-1}\choose{m}}{{d-a_1}\choose{i-|a|}}\frac{(k-1)^{m}}{\ell^m}\ell^{n-1-i}\sigma_{a_1}^{(d)}\tau_{a_2}^{(d)}\mathcal{W}_{\delta_{|a|}^{m}},
$$
in which can be rewritten as
$$
\deg(W)\sum_{|a|=0}^{n-d}\sum_{j=|a|+1}^{n}\sum_{m=0}^{n-d-|a|}(-1)^{\delta_{|a|}^{m}}{{n-j}\choose{m}}{{d-a_1}\choose{j-|a|}-1}\frac{(k-1)^{m}}{\ell^m}\ell^{n-j}\sigma_{a_1}^{(d)}\tau_{a_2}^{(d)}\mathcal{W}_{\delta_{|a|}^{m}}.
$$
However,
$${\psi}_{a}(x)=(1+x)^{d-a_1}x^{n-d-a_2-1}=\sum_{i=0}^{n-|a|-1}{{d-a_1}\choose{i}}x^{n-|a|-i-1}$$
which yields

$$\sum_{j=|a|+1}^{n}{{n-j}\choose{m}}{{d-a_1}\choose{j-|a|}-1}\ell^{n-j-m}=\frac{{\psi}_{a}^{(m)}(x)}{m!}.$$
Therefore,
\begin{equation}
\label{nse}
N(\widetilde{\fol},E)=\sum_{|a|=0}^{n-d}\sum_{m=0}^{n-d-|a|}(-1)^{\delta_{|a|}^{m}}\frac{{\psi}_{a}^{(m)}(\ell)}{m!}(k-1)^m\sigma_{a_1}^{(d)}\tau_{a_2}^{(d)}\mathcal{W}_{\delta_{|a|}^{m}}.
\end{equation}

Now, the formula stated in the theorem is obtained making
$$\sum_{j=1}^{s}\mu(\fol,p_j)=N(\widetilde{\fol},\pnt)-N(\widetilde{\fol},E),$$
that is,
\begin{align*}
&\dps\sum_{j=1}^{s}\mu(\fol,p_j)=1+k+k^2+k^3+\ldots+k^{n-1}+k^n+\\
&\deg(W)\sum_{|a|=0}^{n-d}\sum_{j=|a|}^{n}\sum_{m=0}^{n-d-|a|}(-1)^{\delta_{|a|}^{m}-1}{{n-j}\choose{m}}{{d-a_1}\choose{j-|a|}}(k-1)^{m}\ell^{n-i-m}\sigma_{a_1}^{(d)}\tau_{a_2}^{(d)}\mathcal{W}_{\delta_{|a|}^{m}}^{(d)}.
\end{align*}
\par But, given that
$$\dps\sum_{j=|a|}^{n}{{n-j}\choose{m}}{{d-a_1}\choose{j-|a|}}\ell^{n-j-m}=\frac{\varphi_{a}^{(m)}(\ell)}{m!}$$
with $\varphi_a(x)=x^{n-d-a_2}(1+x)^{d-a_1}$, we get the final
expression given in the Theorem \ref{theorem1}.
\end{proof}
\par Now we consider the case where $W$ is invariant by $\fol$.

\begin{proposition}\label{casel0} Let $\fol$ be a holomorphic foliation by curves on $\pn$ of degree $k$. Assume that $W\subset\pn$ is smooth subvariety of codimension $d\geq 2$ invariant by
$\fol$. Then, the number of isolated singularities (counted with multiplicity) of $\fol$ on $W$ is
$$N(\fol,W)=-\nu(\fol,W)|_{\ell=0}.$$
\begin{proof}
It is enough to observe that
$$\varphi_{a}(x)=x^{n-d-a_2}(1+x)^{d-a_1}=x^{n-d-a_2}+(d-a_1)x^{n-a_1-a_2}+\ldots,$$
thus, $\varphi_{a}^{(m)}(0)=0$ for $m=0,\ldots,n-d-a_2-1$ and
$\varphi_{a}^{(m)}(0)=m!$ for $m=n-d-a_2$.
Now, if $a=(a_1,a_2)$ with $a_1>0$ then $\varphi_{a}^{(m)}(0)=0$ for
$m=0,\ldots,n-d-|a|$, because $m\le n-d-|a|=n-d-a_1-a_2<n-d-a_2$.
Consequently, $\varphi_{a}^{(m)}(0)\ne0$ if, and only if, $a=(a_1,a_2)$
with $a_1=0$ and $m=n-d-|a|$. Writing $|a|=a_2=i$, we get
\begin{align*}
\dps\nu(\fol,W)|_{\ell=0}&= -\dps\sum_{i=1}^{n-d}\deg(W)\tau_{i}^{(d)}(k-1)^{n-d-i}\\
&=-\dps\sum_{i=0}^{n-d}\int_{W}c_i(T_W)\cdot c_1(\ll_{\fol}^*)^{n-d-i} \\
&=-\dps\int_{W}c_{n-d}(T_W\otimes \ll_{\fol}^*)\\
&=-N(\fol,W).
\end{align*}
\end{proof}
\end{proposition}
\section{Proof of Theorem \ref{theorem1}}
It suffices to prove Theorem \ref{theorem1} for a holomorphic
foliation by curves $\fol$ with only an irreducible component $W$
(of positive dimension) of $\sing(\fol)$. In this sense, let $\fol$
be a foliation by curves on $\pn$ of degree $k$ such that
$$
 \mbox{Sing}(\fol) = W \cup\{p_1,\ldots,p_s\},
$$
where $W=Z(h_{1},\ldots,h_{d})$. There exists a one-parameter family
of foliations by curves on $\pn$, given by $\{\fol_t\}_{t\in D}$
where $D=\{t\in\C\,:\,|t|<\epsilon\}$ satisfying conditions
$(i)-(v)$ of Lemma \ref{lema_Theorem1}. When $\ell>0$, we have

$$\dps\sum_{j=1}^{s}\mu(\fol,p_j)=\lim_{t\to0}\sum_{\lim\widetilde{p}_j^t \notin
E}\mu(\widetilde{\fol}_t,\widetilde{p}_j^t),$$ where
$\widetilde{\fol}_t$ is the induced foliation by $\fol_t$ via the
blowup $\pi:\pnt\to\pn$ centered at $W$ with exceptional divisor $E$
and
$\sing(\widetilde{\fol}_t)=\{\widetilde{p}_j^t:\,\,j=1,\ldots,\widetilde{s}_t\}$.
On the other hand,
$$\dps\sum_{j=1}^{s}\mu(\fol,p_j)=N(\widetilde{\fol_t},\pnt) - \lim_{t\to0}\sum_{\lim\widetilde{p}_j^t \in
E}\mu(\widetilde{\fol_t},\widetilde{p}_j^t).$$
Since $\fol_t$ is special along $W$, by Theorem
\ref{nunsing1}, we get
$$\dps\sum_{j=1}^{s}\mu(\fol,p_j)=N(\widetilde{\fol}_t,\pnt) - N(\widetilde{\fol}_t,E) - N(\fol,A_{W}).$$
Then,
$$\dps\sum_{j=1}^{s}\mu(\fol,p_j)=\dps\sum_{i=1}^{n}k^i +\nu(\fol,W) - N(\fol,A_{W}).$$
We finish the proof by considering $\ell=0$. According to Lemma
\ref{lema_Theorem1}, we can suppose that $W$ is invariant by
$\fol_t$. Note that
$$\dps\sum_{j=1}^{s}\mu(\fol,p_j)=\lim_{t\to0}\sum_{\lim p_j^t \notin
W}\mu(\fol_t,p_j^t).$$
Therefore,
$$\dps\sum_{j=1}^{s}\mu(\fol,p_j)=N(\fol_t,\pn) - \lim_{t\to0}\sum_{\lim p_j^t \in
W}\mu(\fol_t,p_j^t).$$
Thus,
$$\dps\sum_{j=1}^{s}\mu(\fol,p_j)=N(\fol_t,\pn) - N(\fol_t,W) - N(\fol,A_{W}).$$
By Proposition \ref{casel0}, we get
$$\dps\sum_{j=1}^{s}\mu(\fol,p_j)=\dps\sum_{i=1}^{n}k^i
+\nu(\fol,W)|_{\ell=0} - N(\fol,A_{W}).$$


\section{Deformation of a foliation by curves in presence of an invariant subvariety }
The aim of this section is to prove an adaption of Lemma
\ref{lema_Theorem1} for foliations by curves  which admit an
invariant subvariety of $\pn$.
\begin{proposition} \label{propfin}
Let $\fol$ be a holomorphic foliation by curves on $\pn$, $n\geq 3$,
of degree $k$ such that $W=Z(h_1,\ldots,h_d)\subset \sing(\fol)$.
Assume that the smooth variety $V=Z(f_1,\ldots,f_m)$ is invariant by
$\fol$ such that $f_j\in\mathcal{I}(W)$ and $1\leq m <d$. If $\fol$
is special along $W$ then
\begin{enumerate}
\item[(i)]
$N(\widetilde{\fol},\widetilde{V})=N(\fol,V)+\dps\sum_{i=0}^{n-d}\sum_{j=0}^{n-d-i}{{n-m-i}\choose{j}}\tau_{i}^{(m)}(k-1)^{j}(-\ell)^{\gamma_i^j}\alpha_{\widetilde{V}_E}^{(\gamma_i^j-1)},
$
\item[(ii)]
$ N(\widetilde{\fol},\widetilde{V}_E)=
\dps\sum_{i=0}^{n-d}\sum_{j=0}^{n-d-i}\sum_{p=i+1}^{n-m-j}{{n-m-p}\choose{j}}(-1)^{\gamma_i^j-1}\ell^{\gamma_p^j}\tau_i^{(m)}(k-1)^j\alpha_{\widetilde{V}_E}^{\gamma_i^j-1},$
\item[(iii)]
$
\nu(\fol,V,W)=N(\fol,V)+\dps\sum_{i=0}^{n-d}\sum_{j=0}^{n-d-i}\frac{\Omega^{(j)}(\ell)}{j!}(-1)^{\gamma_i^j}\ell^{\gamma_p^j}\tau_i^{(m)}(k-1)^j\alpha_{\widetilde{V}_E}^{\gamma_i^j-1},
$
\end{enumerate}
where $\tau_j^{(m)}:=\tau_j^{(m)}(d_1,\ldots,d_q)$, as (\ref{ccti}),
$d_j=\deg(f_j)$, $\gamma_j^i=n-m-i-j$, $\widetilde{\fol}$ is the
strict transform of $\fol$ by the blowup $\pi$ at $W$,
$\widetilde{V}=\overline{\pi^{-1}(V\setminus W)}$,
$\widetilde{V}_E=\widetilde{V}\cap E$,
$\alpha_{\widetilde{V}_E}^{(i)}:=\dps\int_{{V}_E}\pi^*(\mathbf{h})^{n-m-1-i}\cdot
\zeta^i$ and $\Omega(x)=(x^{n-m-i}+x^{n-m-i-1}+\ldots+x^{j})$.
\begin{proof}

\par Let $\alpha_{V}^{(i)}:=\dps\int_{\widetilde{V}}\pi^*(\mathbf{h})^{n-m-i}\cdot \zeta^i$ for $i\ge 0$. It is not difficult to
see that $\alpha_{V}^{(0)}=\deg(V)$ and
$\alpha_{V}^{(i)}=\alpha_{V_E}^{(i-1)}$, for $i\ge 1$. Now, given
that $\alpha_{V_E}^{(i)}=0$ for $i < d-m-1$, then
$$\alpha_{V_E}^{(d-m-1)}=(-1)^{d-m-1}\deg(W).$$ Assume that we have defined $\alpha_{V_E}^{q}$ for $q=0,\ldots,i-1<d-1$. By
(\ref{ccde}), we have the following relation
$$(-1)^{i}{{d}\choose{i}}\zeta^i=c_i(E)-\dps\sum_{|a|=1}^{i}(-1)^{i-|a|}{{d-a_1}\choose{i-|a|}}\rho^*c_{a_1}(N)c_{a_2}(T_{W})\zeta^{i-|a|}.$$
As a consequence
\begin{equation}
\label{ich}
(-1)^{i}{{d}\choose{i}}\alpha_{V_E}^{(i)}={{d}\choose{\varrho_m}}\tau_{i-\varrho_m}^{(d)}\deg(W)-\dps\sum_{|a|=1}^{i}(-1)^{i-|a|}{{d-a_1}\choose{i-|a|}}\sigma_{a_1}^{(d)}\tau_{a_2}^{(d)}\alpha_{V_E}^{(i-|a|)}.
\end{equation}
where $\varrho_m=d-m-1$. Furthermore, in order to compute
$\alpha_{{V}_E}^{(i)}$ for $i\ge d$, we will use (\ref{zetger}).
Again, the computation of $N(\widetilde{\fol},\widetilde{V})$
follows directly from the Baum-Bott formula. Thus,
$$N(\widetilde{\fol},\widetilde{V})=\dps\int_{\widetilde{V}}c_{n-m}(T_{\widetilde{V}}\otimes
\ll_{\widetilde{\fol}}^{*})=\dps\int_{\widetilde{V}}\sum_{i=0}^{n-m}c_i(T_{\widetilde{V}})c_1(\ll_{\widetilde{\fol}}^{*})^{n-m-i}$$
where $c_i(T_{\widetilde{V}})=\pi^*c_i(T_{V})$. As before, we obtain

\begin{equation*}
\dps
N(\widetilde{\fol},\widetilde{V})=\sum_{i=0}^{n-m}\sum_{j=0}^{n-m-i}\int_{\widetilde{V}}{{n-m-i}\choose{j}}\pi^*c_i(T_{V})\pi^*c_1(\ll_{{\fol}}^{*})^{j}(-\ell\zeta)^{\gamma_j^i}.
\end{equation*}
For $j=n-m-i$ and $i+j\le n-d$ in the last sum, we have concluded
the proof of (i).

\noindent To prove $(ii)$, it is enough to observe that
$\widetilde{V}$ and the exceptional divisor $E$ are nonsingular
subvarieties of a nonsingular complex variety ${\pnt}$. Thus, given
that $\widetilde{V}$ and $E$ intersect properly and transversally,
the intersection $\widetilde{V}_{E}$ is nonsingular and an
elementary Chern class computation proves that
$$ c(T_E)\cdot c(\widetilde{V}) = c(T_{\pnt})\cdot c(T_{\widetilde{V}_E}).$$
As consequence, we obtain that
$$c_q(T_{\widetilde{V}_E})=\dps\sum_{i=0}^{q}(-1)^{q-i}c_i(T_{\widetilde{V}})\cdot
(\zeta)^{q-i}.$$ To conclude the proof of (2), it is enough to
consider
$$N(\widetilde{\fol},\widetilde{V}_E)= \dps
\int_{\widetilde{V}_E}c_{n-m-1}(T_{\widetilde{V}_E}\otimes\ll_{\widetilde{\fol}}^{*})=\int_{\widetilde{V}_E}\sum_{q=0}^{n-m-1}c_q(T_{\widetilde{V}_E})c_1(\ll_{\widetilde{\fol}}^{*})^{n-m-1-q}.$$
Now, we need to repeat the same steps,

$$N(\widetilde{\fol},\widetilde{V}_E)= \dps
\sum_{q=0}^{n-m-1}\sum_{i=0}^{q}\sum_{j=0}^{n-m-1-q}{{n-m-1-q}\choose{j}}(-1)^{\gamma_i^j-1}\ell^{\gamma_q^j-1}\tau_i^{(m)}(k-1)^j\alpha_{\widetilde{V}_E}^{\gamma_i^j-1}.$$
Thus,

$$N(\widetilde{\fol},\widetilde{V}_E)=
\dps\sum_{i=0}^{n-m-1}\sum_{j=0}^{n-m-1-i}\sum_{q=i}^{n-m-1-j}{{n-m-1-q}\choose{j}}(-1)^{\gamma_i^j-1}\ell^{\gamma_q^j-1}\tau_i^{(m)}(k-1)^j\alpha_{\widetilde{V}_E}^{\gamma_i^j-1}.$$
Making $q:=p-1$ and considering $i+j\le n-d$, we conclude the proof
of (ii). Finally, to proof (iii) it is enough to observe that
$\nu(\fol,V,W)=N(\fol,V\setminus
W)=N(\widetilde{\fol},\widetilde{V})-N(\widetilde{\fol},\widetilde{V}_E).$
Therefore,
$\nu(\fol,V,W)=N(\fol,V)+\dps\sum_{i=0}^{n-d}\sum_{j=0}^{n-d-i}\sum_{p=i}^{n-d-j}{{n-m-p}\choose{j}}(-1)^{\gamma_i^j}\ell^{\gamma_p^j}\tau_i^{(m)}(k-1)^j\alpha_{\widetilde{V}_E}^{\gamma_i^j-1}
$. But,
$$\sum_{p=i}^{n-m-j}{{n-m-p}\choose{j}}\ell^{\gamma_p^j}=\dps\frac{1}{j!}\frac{d^{j}}{dx^j}(x^{n-m-i}+x^{n-m-i-1}+\ldots+x^{j})\bigg|_{x=\ell}$$
which concludes the proof of (iii).
\end{proof}
\end{proposition}
\begin{remark} The proof of Proposition \ref{propfin} may be obtained by Theorem \ref{theorem1} replacing $\pn$ by $V$. Indeed, let $\pi_V: \widetilde{V}\to V$
be the blowup of $V$ centered at $W$ with $E_V$ the exceptional
divisor. Naturally, $E_V=E\cap V=V_{E}$,
$\zeta_{V}=\zeta|_{V}:=c_1(N_{V_E/\widetilde{V}})$ and
$\pi_V=\pi|_{V}$. Furthermore,
 if we consider $\ell=0$ then we obtain
$\nu(\fol,V,W)= N(\fol,V)-N(\fol,W)$. In fact, for the case where
$d-m=1$ we have that $\zeta_V - \pi^*(c_1(N_{W/V})=0$. See
(\ref{equdch}). Therefore,
$$N(\widetilde{\fol},\widetilde{V}_E)= \dps\sum_{i=0}^{n-d}\int_{\widetilde{V}_E}\pi_V^*c_i(T_{\widetilde{V}_E})\pi_V^*c_1(\ll_{{\fol}}^{*})^{n-d-i}.$$
By induction, we can prove that
$$c_i(T_{\widetilde{V}_E})=\sum_{j=0}^{i}=(-1)^{i-j}\zeta_V^{i-j}\pi_V^*c_j(T_V)=\pi_V^*c_i(T_W)$$
because $c_i(T_V)|_{W}=c_1(N_{W/V})|_Wc_{i-1}(T_W)+c_{i}(T_W)$. As a
consequence, if $d-m=1$ then
$$N(\widetilde{\fol},\widetilde{V}_E)=\sum_{i=0}^{n-d}\int_{W}c_i(T_W)\cdot \pi_V^*c_1(\ll_{\fol}^{*})^{n-d-i}=N(\fol,W).$$
For the general case where $q=d-m\geq 2$, again by (\ref{equdch}) we
have that
$$\zeta_V^q -
\pi_V^*c_1(N_{W/V})+\ldots+(-1)^{q}\pi_V^*c_q(N_{W/V})=0,$$
$\alpha_{\widetilde{V}_E}^{(i)}=\int_{\widetilde{V}_E}
\pi_V^*(\mathbf{h})^{n-m-1-i}\zeta_V^{i}=0$ for $i=0,\ldots,q-2$ and
$\alpha_{\widetilde{V}_E}^{(q-1)}=(-1)^{q-1}\deg(W)$ which yields
$$N(\widetilde{\fol},\widetilde{V}_E)=\int_{\widetilde{V}_E}\sum_{i=q-1}^{n-m-1}\sum_{j=0}^{i-q+1}(-1)^{i-j}\zeta_{V}^{i-j}\pi_V^*c_j(T_V)\cdot \pi_V^*c_1(\ll_{{\fol}}^{*})^{n-m-1-i}.$$

Fixing $i\in\{q-1,\ldots,n-m-1\}$, by induction we prove that
$$\sum_{j=0}^{i-q+1}(-1)^{i-j}\zeta_{V}^{i-j}\pi_V^*c_j(T_V)=(-1)^{q-1}\pi_V^*c_{i-q+1}(T_W)\zeta_V^{q-1}+\ldots$$
which yields

$$N(\widetilde{\fol},\widetilde{V}_E)=\sum_{i=q-1}^{n-m-1}\tau_{i-q+1}^{(d)}(k-1)^{n-d-i+q-1}deg(W)=N(\fol,W)$$
for $\ell=0$.

\end{remark}

\section{Proof of Theorem \ref{principal}}
Let $\fol$ be a foliation by curves on $\pn$ of degree $k$ such that
$$
 \mbox{Sing}(\fol) = W \cup\{p_1,\ldots,p_s\},
$$
where $W=Z(h_{1},\ldots,h_{d})$ and $$V=Z(f_1,\ldots,f_m)$$ is
invariant by $\fol$ such that $f_i\in\mathcal{I}(W)$ for all $1\leq
i\leq m$. It is clear that $W\subset V$. By Lemma
\ref{lema_Theorem1}, there exists a one-parameter family of
holomorphic foliations by curves on $\pn$ given by $\{\fol_t\}_{t\in
D}$ where $D=\{t\in\C\,:\,|t|<\epsilon\}$ satisfying conditions
$(i)-(vi)$.
\par We will consider two cases,
$\ell=0$ and $\ell>0$.
If $\ell=0$, then $W$ and $V$ are $\fol_t$-invariant for all $t\in
D\setminus\{0\}$. As consequence,
$$\dps\sum_{p_j\in V}\mu(\fol,p_j)=\lim_{t\to0}\sum_{\lim p_i^t\in V}\mu(\fol_t,p_i^t)-\lim_{t\to0}\sum_{\lim p_i^t\in W}\mu(\fol_t,p_i^t).$$
Thus,
$$\dps\sum_{p_j\in V}\mu(\fol,p_j)=N(\fol_t,V)+N(\fol,A_{V})-N(\fol_t,W)-N(\fol,A_{W}).$$
Therefore,
$$\dps\sum_{p_j\in V}\mu(\fol,p_j)=N(\fol_t,V)-N(\fol,W)+N(\fol,A_{V\setminus W})$$
with $N(\fol,A_{V\setminus W})\ge0.$ Now, in the case where $\ell>0$ we get

$$\dps\sum_{p_j\in V}\mu(\fol,p_j)=\lim_{t\to0}\sum_{\lim\widetilde{p}_i^t\in \widetilde{V}}\mu(\widetilde{\fol_t},\widetilde{p}_i^t) - \lim_{t\to0}\sum_{\lim\widetilde{p}_j^t \in
\widetilde{V}_E}\mu(\widetilde{\fol_t},\widetilde{p}_j^t)$$ where
$\widetilde{V}$ is the strict transform of $V$ via $\pi$ and
$\widetilde{V}_E=\widetilde{V}\cap E$.
Then,
$$\dps\sum_{p_j\in V}\mu(\fol,p_j)=N(\widetilde{\fol_t},\widetilde{V})+N(\widetilde{\fol},A_{\widetilde{V}})-N(\widetilde{\fol_t},\widetilde{V}_E)-N(\widetilde{\fol},A_{\widetilde{V}_E}).$$
So, given that $N(\fol,A_W) \ge
N(\widetilde{\fol},A_{\widetilde{V}_E})$, we get

$$\dps\sum_{p_j\in V}\mu(\fol,p_j)\le N(\widetilde{\fol},\widetilde{V})-N(\widetilde{\fol},\widetilde{V}_E)+N(\widetilde{\fol},A_{V\setminus W}),$$
where $N(\widetilde{\fol},\widetilde{V})$ and
$N(\widetilde{\fol},\widetilde{V}_E)$ are given in Proposition
\ref{propfin}. Since
$$N(\tilde{\fol},\tilde{V})-N(\tilde{\fol},\tilde{V}_E)=N(\fol,V)+
\nu(\fol,V,W),$$ we conclude $$\sum_{p\in V\setminus
W}\mu(\fol,p)\le N(\fol,V)+\nu(\fol,V,W)+N(\fol,A_{V\setminus W}).$$
Finally, note that as in Theorem \ref{theorem1}, these expressions
coincide if $\ell=0$.

\section{Explicit computations for a family of foliations by curves }
\begin{lemma}\label{lema5} Let $\fol$ be a
foliation by curves on $\pn$ of degree $k$ such that $W\subset\sing(\fol)$
with $W=V(f_1,\ldots,f_d)$ and $k_j=\deg(f_j)=1$ for $i=1,\ldots,d$.
Assume that $\ell=mult_E(\fol)=k-1$. Then
\begin{enumerate}
\item[(i)] $\nu(\fol,W)=-(k^d+\ldots+k^n)$,
\item[(ii)] $N(\widetilde{\fol},E)=(n+1-d)(1+k+\ldots+k^{d-1}).$
\end{enumerate}
\end{lemma}
\begin{proof}
Given that $k_1=\ldots=k_d=1$, $c(N_{W/\pn})=\sum
\sigma_i^{(d)}\mathbf{h}^i$, $\sigma_i^{(d)}={{d}\choose{i}}$,
$c(T_W)=\sum \tau_i^{(d)}\mathbf{h}^i$,
$\tau_i^{(d)}={{n+1-d}\choose{i}}$ and
$\mathcal{W}^{(d)}_{\delta}:=\mathcal{W}^{(d)}_{\delta}(k_1,\ldots,k_{d})={{\delta+d-1}\choose{d-1}}$.
To prove (i) it is enough to observe that

$$\sum_{a_1+a_2=j}\varphi_{a}(x)\sigma_{a_1}^{(d)}\tau_{a_2}^{(d)}=\frac{\Psi_{n}^{(j)}(x)}{x\cdot
j!}$$ where $\Psi_{n}(x)=(1+x)^d\cdot x^{n+1-d}$. Then,

$$\nu_n(\fol,W)=
-\sum_{m=0}^{n-d}\sum_{j=0}^{n-d-m}\sum_{i=0}^{m}(-1)^{\delta_{i}^{j}}\frac{\ell^{i-1}\Psi_{n}^{(i+j)}(\ell)}{i!j!}{{n-m-j-1}\choose{d-1}},$$
with $\delta_{i}^{j}=n-d-i-j$.  For fixed $d$, we obtain the
relation

$$\nu_{n+1}(\fol,W)=\nu_n(\fol,W)-\sum_{i=0}^{n+1-d}\sum_{j=0}^{n+1-d-i}\frac{\ell^{i}\Psi_{n}^{(i+j)}(\ell)}{i!j!}{{n-i-j}\choose{d-1}}.$$
Result (i) follows because $\nu_d(\fol,W)=k^d$ and
$$\sum_{i=0}^{n+1-d}\sum_{j=0}^{n+1-d-i}\frac{\ell^{i}\Psi_{n}^{(i+j)}(\ell)}{i!j!}{{n-i-j}\choose{d-1}}=(1+\ell)^{n+1}=k^{n+1}.$$
The proof of (ii) is done in a similar way using (\ref{nse}).
\end{proof}

\begin{example} Let $\fol_k$ be the holomorphic foliation by curves
of degree $k$ defined in $\pn$, $n\ge3$, such that in the affine
open $U_n=\{[\xi_0:\cdots:\xi_n]\in \pn, \xi_n\ne 0\}$ is described
by the following vector field
$$X_{\fol_k}=\sum_{i=1}^{n}\sum_{|a|=m_i}\vartheta_{a,i}(x)x_1^{a_1}\cdots x_d^{a_d}\frac{\partial}{\partial x_i}:=\sum_{i=1}^{n}P_i(x)\frac{\partial}{\partial x_i},
\quad x_i=\frac{\xi_{i-1}}{\xi_n}$$ where $2\le d\le n-1$,
$a=(a_1,\ldots, a_d)$, $m_1=\ldots=m_d=m_{d+1}+1=\ldots=m_{n}+1=k$,
$\vartheta_{a,i}$ is a constant for  $i=1,\ldots,d$ and
$\vartheta_{a,i}$ is an affine linear function for $i=d+1,\ldots,n$.
If $k\ge2$ then the set $W_0\cong\P^{n-d}$ given by
$\xi_0=\ldots=\xi_{d-1}=0$ is contained in the singular locus of
$\fol_k$. Otherwise, if $k=1$ then $W_0$ is an invariant set by
$\fol_k$ since $m_i=0$, for $d < i \leq n$. Let us assume that
$$W_0=Z(P_1(\xi),\ldots,P_d(\xi)).$$
This hypothesis is essential in order that $W_0$ be an irreducible
component of $\sing(\fol_k)$. As before, for $k\ge2$, let
$\pi:\pnt\to\pn$ be the blowup of $\pn$ at $W_0$ with $E$ is the
exceptional divisor. It is not difficult to see that $\deg(W_0)=1$
and $\ell=m_{E}(\fol_k,W_0)=k-1$.

If $\sing(\fol_k)=W_0\cup\{p_1,\ldots,p_r\}$ then
\begin{equation}\label{ns}
\sum_{j=1}^{r}\mu(\fol_k,p_j)=1+k+\ldots+k^{d-1}\end{equation} for
all $k\ge 2$. In fact, for $k\ge2$ and $n=d+1$, i.e.; $W_0$ is a
curve, we will consider $\pn=U_n\cup H_n^{\infty}$ where
$H_n^{\infty}$ is the hyperplane at infinity. Let us define
$\fol_k^{(1)}=\fol_k|_{H_n^{\infty}}$ and $p_{\infty}=W_0\cap
H_n^{\infty}$. It is not difficult to see that
$H_n^{\infty}\cong\P^{n-1}$ is an invariant set of $\fol_k$ and
$\fol_k^{(1)}$ is a holomorphic foliation in $H_n^{\infty}$ of
degree $k$ too. Furthermore, given that
$\mu(\fol_k^{(1)},p_{\infty})=k^{n-1}$,

$$\sum_{j=1}^{r}\mu(\fol_k,p_j)=\sum_{j=1}^{r}\mu(\fol_k^{(1)},p_j)=\sum_{j=0}^{n-1}k^j-k^{n-1}=\sum_{j=0}^{d-1}k^j.$$
Therefore, (\ref{ns}) is true for $n=d+1$. Now, by the hypothesis of
induction let us assume that the result is true for $n>d$ where $d$
is fixed. Let us consider $\fol_k$ defined in $\P^{n+1}$ and special
along $W_0\subset\P^{n+1}$ with $cod(W_0)=d$. Thus,
$\P^{n+1}=U_{n+1}\cup H_{n+1}^{\infty}$. Again,
$H_{n+1}^{\infty}\cong \pn$ is an invariant set of $\fol_k$.
Defining $\fol_k^{(1)}=\fol_k|_{H_{n+1}^{\infty}}$ and $W_1=W_0\cap
H_{n+1}^{\infty}$ we obtain a holomorphic foliation by curves in
$H_{n+1}^{\infty}$ of degree $k$ and special along $W_1$ too. Given
that $W_1$ in $H_{n+1}^{\infty}$ also has the codimension $d$, we
get
$$\sum_{j=1}^{r}\mu(\fol_k,p_j)=\sum_{j=1}^{r}\mu(\fol_k^{(1)},p_j)=1+k+\ldots+k^{d-1}$$
by hypothesis of induction. This number agrees with the Theorem
\ref{theorem1} taking $\ell=k-1$ and $\deg(W_0)=1$. See the Lemma
\ref{lema5}. For $k=1$, the equation (\ref{ns}) may be interpreted
as follows: let us consider the singular points $p_j\notin W_0$ then

$$\sum_{p_j\in\pn\setminus W_0}\mu(\fol_k,p_j)=\sum_{p_j\in\pn}\mu(\fol_k,p_j)-\sum_{p_j\in
W_0}\mu(\fol_k,p_i).$$

Given that $W_0\cong\P^{n-d}$ is invariant by $\fol_k$, we get

$$\sum_{p_j\notin W_0}\mu(\fol_k,p_j)=n+1-(n-d+1)=d.$$

Again for $k\ge2$, by the Lemma (\ref{lema5}), we obtain that
\begin{equation}\label{nts}
N(\widetilde{\fol}_k,\pnt)=\sum_{j=1}^{r}\mu(\fol_k,p_j)+N(\widetilde{\fol}_k,E)=(n+2-d)\sum_{j=0}^{d-1}k^j.
\end{equation}

Now, let us consider $V_{\lambda}\subset\pn$ a $(n-d+1)$-plane
defined by $V_{\lambda}=\{\xi_j=\lambda_j\xi_0, j=1,\ldots,d\}$,
$\lambda=(\lambda_1,\ldots,\lambda_d)\in\C^{d-1}$. Thus, $W_0\subset
V_{\lambda}$ and $cod(V_{\lambda})=d-1$. In order to $V_{\lambda}$
be invariant by $\fol_k$ the parameter $\lambda$ must be a singular
point of a holomorphic foliation by curves of degree $k$ on
$\P^{d-1}$. Consequently, the number of $(n-d+1)$-plane
$V_{\lambda}$ invariant by $\fol_k$ is $1+k+\ldots+k^{d-1}$. Let $V$
one of them. By (\ref{nts}), we obtain that
$N(\widetilde{\fol}_k,\widetilde{V})=(n+2-d)$ where $\widetilde{V}$
is the strict transform of $V$ by $\pi^{-1}$. Furthermore, by Lemma
(\ref{lema_Theorem1}) we get

$$N(\widetilde{\fol}_k,\widetilde{V}\cap E)=n+1-d.$$
These numbers agree with Proposition (\ref{propfin}) taking
$\deg(V)=1$ and $\ell=k-1$ since
$\alpha_{V}^{(0)}:=\dps\int_{\widetilde{V}}\pi^*(\mathbf{h})^{n-d}=1$
and
$\alpha_{V_E}^{(i)}:=\int_{V_E}\pi^*(\mathbf{h})^{n-d-i-1}\cdot\zeta^i=1$
for $i\ge0$.

\noindent Now, let $\mathcal{G}_k$ be an one-dimensional holomorphic
foliation of degree $k$ in $\pn$ also described in $U_n$ by the
vector field
\begin{equation}\label{camg}
X_{\mathcal{G}_k}=\sum_{i=1}^{n}\sum_{|a|=k}c_{a,i}x_1^{a_1}\cdots
x_d^{a_d}\frac{\partial}{\partial x_i}=
\sum_{i=1}^{n}Q_i(x)\frac{\partial}{\partial x_i}
\end{equation}
where $x_i=\xi_{i-1}/\xi_n$, $a=(a_1,\ldots,a_d)$ and $c_{a,i}$ a
constant. Thus, $W_0\subset \sing(\mathcal{G}_k)$ and
$\ell=m_{E}(\pi^*\mathcal{G}_k)=k-1$. We have two different
situations according to $\ell=0$ or $\ell \ge 1$. We start our
analysis considering $\ell=0$, i.e.; $\mathcal{G}_k$ has degree one.
Let $\fol_t$ be a holomorphic foliation on $\pn$ given in the Lemma
\ref{lema_Theorem1}, for $t\in D\setminus\{0\}$, $D=D(0,\epsilon)$,
with $\epsilon$ sufficiently small,

$$X_{\fol_t}=\sum_{i=1}^{n}\sum_{j=1}^{d}c_{i,j}x_j\frac{\partial}{\partial
x_i}+t\sum_{i=1}^{n}Y_i(x)\frac{\partial}{\partial x_i}$$ where

$$Y_i=\left\{\begin{array}{ll}
              \sum_{j=1}^{d}b_{ij}x_j,& \mathrm{for}\quad 1\le i\le
              d\cr
              b_{i0}+\sum_{j=1}^{n}b_{ij}x_j,& \mathrm{for}\quad d+1\le i\le
              n
              \end{array}\right.$$
\end{example}
\noindent with $b_{ij}\in\C$. Let $A_t=(a_{ij}^{t})_{1\le i,j\le d}
\in M_{\C}(d,d)$ be $d\times d$ matrix given by
$a_{ij}^{t}=c_{ij}+tb_{ij}$. We can admit $\mathrm{Det}(A_t) \ne 0$
and $A_t$ has $d$ distinct eigenvalues for all $t\in
D\setminus\{0\}$ which yields that the only singular point $p_n \in
\sing(\fol_t)\cap U_n$ is given by
$p_n=(0,\ldots,0,p_{d+1,n},\ldots,p_{n,n})$ where
$(p_{d+1,n},\ldots,p_{n,n})$ is the solution of linear system
$b_{i0}+\sum_{j=d+1}^{n}b_{ij}x_j=0$ for $d+1\le i\le n$, for all
$t\in D\setminus\{0\}$. In the affine open $U_{n-1}\subset \pn$ with
coordinate $(y_i)$, $y_i={\xi_{i-1}}/{\xi_{n-1}}$ for $1\le i \le
n-1$ and $y_n={\xi_{n}}/{\xi_{n-1}}$, $\fol_t$ is described by the
vector field

$$X_{\fol_t}=\sum_{i=1}^{n}Q_i^t(y)\frac{\partial}{\partial
y_i}$$ where

$$Q_i^t(y)=\left\{\begin{array}{ll}
               \sum_{j=1}^{d}(c_{ij}+tb_{ij})y_j-y_ig_n(y),& \mathrm{for}\quad 1\le i\le
              d\cr
               g_i(y)-y_ig_n(y),& \mathrm{for}\quad d+1\le i\le
              n-1\cr
              -y_ng_n(y),& \mathrm{for}\quad i=n
              \end{array}\right.
$$
with
$g_i(y)=\sum_{j=1}^{d}c_{ij}y_j+t(b_{i0}y_n+b_{in}+\sum_{j=1}^{n-1}b_{ij}y_j$.
In order to determine the singular points of $\fol_t$ we need to
consider $y_n=0$ or $g_n(y)=0$. It is not possible to have these two
conditions satisfied simultaneously because they would result in an
inconsistent linear system. Namely, $g_i(y)=0$, for $d+1\leq i \le
n$ and $y_1=\ldots=y_d=y_n=0$, i.e.; the number of equations is
greater than the number of unknown variables.  Note that condition
$g_n(y)=0$ results in the singular point $p_n\in U_{n}\cap U_{n-1}$.
Therefore, we can only consider $y_n=0$ and $g_n(y)\ne 0$. For all
$t\in D\setminus\{0\}$, the point $p_{n-1}\in \sing(\fol_t)\cap
U_{n-1}$ where
$p_{n-1}=(0,\ldots,0,p_{n-1,d+1},\ldots,p_{n-1,n-1},0)$ such that
$(p_{n-1,d+1},\ldots,p_{n-1,n-1})$ is the solution of the linear
system $b_{in}+\sum_{j=d+1}^{n-1}b_{ij}y_j=0$ for $d+1\le i \le
n-1$. This way, in each an affine open $U_i$, for $i=d,\ldots, n$ we
get one singular point $p_i\in W_0\cap U_i$. Therefore,
$\mu(\fol_k,W_0)\ge n-d+1$. The other singular points
$p_m^t=(y_1^t,\ldots,y_{n-1}^t,0)$ of $\fol_t$ are obtained as
follows: $(y_1^t,\ldots,y_{d}^t)$ is the eigenvector of $A_t$
associated to eigenvalue $\lambda_m^t$. Thus,
$Q_i^t(p_m^t)=y_i^t(\lambda_m^t-g_n(p_m^t))=0$ for $i=1,\ldots,d$.
which gives $\lambda_m^t=g_n(p_m^t)$. The other components of
$p_m^t$ are obtained from the linear system
$\sum_{j=1}^{d}c_{ij}y_j^t+t(b_{in}+\sum_{j=1}^{n-1}b_{ij}y_j^t)=0$,
for $d+1\le i \le n-1$. Without loss of generality, we can suppose
that $y_1^t\ne 0$ which results in

$$y_1^t = \frac{\lambda_m^t -
tb_{nn}}{\sum_{j=1}^{d}c_{nj}u_j^t+t(\sum_{j=1}^{n-1}b_{nj}u_j^t)}$$
where $u_j^t = y_j^t/y_1^t$.

Consequently, we obtain $d$ points on $p_m^t\in\pn\setminus W_0$,
counted with multiplicity. However,

$$\lim_{t\to0}y_1^t =\frac{\lambda_m^0}{c_{n1}+\sum_{j=2}^{d}c_{nj}u_j^0}$$
where $\lambda_m^0$ is a eigenvalue of $A_0$. Therefore,
$p_{n-1}=\lim_{t\to0}p_m^t$ is an embedded point of $W_0$ if only if
$\lambda_m^0=0$, i.e.; $\mathrm{det}(A_0)=0$.  Let $P_{A_0}(\lambda)
=Det(A_0-\lambda I)=\lambda^{q}p_0(\lambda)$ with $q\ge 0$ and
$p_0(0)\ne0$. It follows that $W_0$ has $q$ embedded closed points,
counted with multiplicity. However, $q\le d-1$, we have that
\begin{equation}
\label{ulte} n-d+1 \le \mu(\mathcal{G}_k,W_0)=n-d+1+q \le n
\end{equation}
for $k=1$. At this point, we will analyze the case where $\ell\ge
1$. At now, $W_0$ is a nondicritical of type (iii) component of
$\sing(\mathcal{G}_k)$. Let $\mathcal{G}_t$ be a holomorphic
foliation by curves as given in Lemma \ref{lema_Theorem1}. Again, in
the affine open $U_n$, $\mathcal{G}_t$ is described by the following
vector field

$$X_{\mathcal{G}_t}=\sum_{i=1}^{n}Q_i(x)\frac{\partial}{\partial
x_i}+t\sum_{i=1}^{n}G_i(x)\frac{\partial}{\partial
x_i}=\sum_{i=1}^{n}G_i^t(x)\frac{\partial}{\partial x_i}$$ where

$$G_i(x)=\left\{\begin{array}{ll}
\dps\sum_{|a|=k}b_{a,i}x_1^{a_1}\cdots x_d^{a_d},&\mathrm{for}\quad
1\le i\le d\cr \dps\sum_{|a|=k-1}\upsilon_{a,i}(x)x_1^{a_1}\cdots
x_d^{a_d},&\mathrm{for}\quad d+1\le i\le n
                 \end{array}\right.
                 $$
with
$\upsilon_{a,i}(x)=\varrho_{a,i,0}+\sum_{j=1}^{n}\varrho_{a,i,j}\cdot
x_j$ is an affine linear function and $a=(a_1,\ldots,a_d)$.

Even if the foliation $\mathcal{G}_k$ has an invariant variety
$V=Z(f_1,\ldots,f_d)$ such that $f_i\in\mathcal{I}(W)$, we can
choose the parameters $b_{a,i}$ in order that $\mathcal{G}_t$ also
has $V$ as invariant set and $W_0=Z(G_1^t(\xi),\ldots,G_d^t(\xi))$
for $t\in D\setminus\{0\}$.

In the chart $\widetilde{U}_1$ with coordinate $\varsigma\in\C^n$
such that $\sigma_1(u)=x$, we obtained the induced foliation
$\widetilde{\mathcal{G}}_t$  by $\mathcal{G}_k$ via blowup $\pi$ as
follows
$$X_{\widetilde{\mathcal{G}}_t}=u_1\widetilde{G}_1^t(u)\frac{\partial}{\partial
u_1}+\sum_{i=2}^{d}[\widetilde{G}_i^t(u)-u_i\widetilde{G}_1^t(u)]\frac{\partial}{\partial
u_i}+\sum_{i=d+1}^{n}\widetilde{G}_i^t(u)\frac{\partial}{\partial
u_i}$$ where

$$\widetilde{G}_i^t(u)=\left\{\begin{array}{ll}
                               \dps\sum_{|a|=k}(c_{a,i}+tb_{a,i})u_2^{a_2}\cdots u_d^{a_d},&\mathrm{for}\quad
                               i=1,\ldots,d\cr
                               t\dps\sum_{|a|=k-1}\widetilde{\upsilon}_{a,i}(u)u_2^{a_2}\cdots u_d^{a_d}+u_1\widetilde{Q}_i(u),&\mathrm{for}\quad
                               i=d+1,\ldots,n,
                               \end{array}\right.
$$
with
$\widetilde{\upsilon}_{a,i}(u)=\varrho_{a,i,0}+\sum_{j=d+1}^{n}\varrho_{a,i,j}u_j$
and some $\widetilde{Q}_i$. For $t\in D\setminus\{0\}$, it is not
difficult to see that there exist $1+k+\ldots+k^{d-1}$ singular
points, counted with multiplicities, on $\widetilde{U}_1\cap E$.
Note that the exceptional divisor $E$ is defined in
$\widetilde{U}_1$ by $u_1=0$ and the system
$\widetilde{G}_i^t(u)-u_i\widetilde{G}_1^t(u)=0$, for $i=2,\ldots,d$
has $1+k+\ldots+k^{d-1}$ solutions, counted with multiplicities.
Furthermore,
$\sing(\widetilde{\mathcal{G}}_t)\cap\widetilde{U}_1\setminus
E=\emptyset$ because $W_0=Z(Q_i(\xi)+tY_i(\xi))$, $i=1,\ldots,d$; $t\neq 0$.

In the the affine open $U_{n-1}\subset \pn$, with coordinate
$y=(y_i)\in \C^n$, with $y_i=x_i/x_n$ for $1\le i \le n-1$ and
$y_n=1/x_n$, we get

$$X_{\mathcal{G}_t}=\sum_{i=1}^{d}\big(G_i^t(y)-y_iH_n(y)\big)\frac{\partial}{\partial y_i}+\sum_{i=d+1}^{n-1}\big(H_i(y)-y_iH_n(y)\big)\frac{\partial}{\partial y_i}-y_nH_n(y)\frac{\partial}{\partial y_n} $$
where
$$H_i(y)=t\sum_{|a|=k-1}\widetilde{\kappa}_{a,i}(y)y_1^{a_1}\cdots y_d^{a_d}+ \widetilde{Q}_i(y),$$
and
$\widetilde{\kappa}_{a,i}(y)=\varrho_{a,i,0}y_n+\varrho_{a,i,n}+\sum_{j=1}^{n-1}\varrho_{a,i,j}\cdot
y_j$. Again, in the chart $\widetilde{U}_1$ with coordinate
$u\in\C^n$ such that $y=\sigma_1(u)$ we get the vector field which
describes the induced foliation $\widetilde{\mathcal{G}}_t$ in
$\widetilde{U}_1$ as follows
$$
\begin{array}{ll}
\dps X_{\widetilde{\mathcal{G}}_t}=&
u_1\big(\widetilde{G}_1^t(u)-\widetilde{H}_n(u)\big)\frac{\partial}{\partial
u_1}+\dps\sum_{i=2}^{d}\big(\widetilde{G}_i^t(u)-u_i\widetilde{G}_1^t(u)\big)\frac{\partial}{\partial
u_i}+\cr
&\dps\sum_{i=d+1}^{n-1}\big(\widetilde{H}_i^t(u)-u_i\widetilde{H}_1^t(u)\big)\frac{\partial}{\partial
u_i}-u_n\widetilde{H}_n(u)\frac{\partial}{\partial u_n}
\end{array}
$$
where
$$\widetilde{G}_i^t(u)=\sum_{|a|=k}(c_{a,i}+tb_{a,i})u_2^{a_2}\ldots u_d^{a_d}$$

\noindent and

$$\widetilde{H}_i^t(u)=u_1\sum_{|a|=k}(c_{a,i})u_2^{a_2}\ldots u_d^{a_d}+t\sum_{|a|=k-1}\widetilde{\kappa}_{a,i}(\sigma_1(u))u_2^{a_2}\ldots u_d^{a_d}.$$

In order to determine the singular points of $\widetilde{G}_t$ in
$\widetilde{U}_1$, we need to analyze two conditions $u_n=0$ and
$\widetilde{H}_n(u)=0$. If $u_n=0$ and $\widetilde{H}_n(u)=0$ then
we would have a inconsistent system, a namely,
$\widetilde{H}_i(u)=0$ for $i=d+1,\ldots,n$ i.e.; more equations
than unknown variables. If $u_n\ne0$ and $\widetilde{H}_n(u)=0$ we
would have the same singularities obtained when we analysis the
foliation $\widetilde{G}_t$ in $U_n$. Therefore, it is enough to
consider only the case where $u_n=0$.

On the exceptional divisor $E$ which on $\widetilde{U}_1$ is defined
by $u_1=0$ we obtain $1+k+\ldots+k^{d-1}$ more singular points,
counted with multiplicities. In fact, the system
$\widetilde{G}_i^t(u)-u_i\widetilde{G}_1^t(u)=0$, for $i=2,\ldots,d$
admits $1+k+\ldots+k^{d-1}$. Thus, with each solution obtained, we
will solve the system $\widetilde{H}_i(u)-u_i\widetilde{H}_n(u)=0$
for $i=d+1,\ldots,n-1$. Consequently, after we consider each affine
open $U_i$, for $i=d,\ldots,n$, we obtain the total number
$(n-d+1)(1+k\ldots+k^{d-1})$ of singularities on $E$ which agrees
with Theorem \ref{theorem1}. Again, see Lemma \ref{lema5}. The
singular points $\widetilde{p}_m^t$ that are not contained in the
exceptional divisor $E$ are obtained as follows:
$\widetilde{p}_m^t=(u_1^t,\ldots,u_{n-1},0)$ where
$(u_2^t,\ldots,u_d^t)$ is a solution of the system
$\widetilde{G}_i^t(u)-u_i\widetilde{G}_1^t(u)=0$, for $i=2,\ldots,d$
and

$$u_1^t =
\frac{\dps\widetilde{G}_1^t(u)-t\sum_{|a|=k-1}\widetilde{\kappa}_{a,i}(\sigma_1(u))u_2^{a_2}\ldots
u_d^{a_d}}{\dps\sum_{|a|=k}c_{a,n}u_2^{a_2}\ldots u_d^{a_d}}.$$
Thus, after we obtain $(u_i^t)$, for $i=1,\ldots,d$ we get the
others components of $\widetilde{p}_m^t$ solving the system
$H_i^t-u_iH_n(u)=0$ for $i=d+1\ldots,n-1$.

Therefore, $p_m^t=\sigma_1(\widetilde{p}_m^t)$ is a singular point
of $\widetilde{G}_t$. By the same argument, if $t\in
D\setminus\{0\}$ then the foliation $\widetilde{\mathcal{G}}_t$
admits $1+k+\ldots+k^{d-1}$ singular isolated points, counted with
multiplicities.

$$\mu(\mathcal{G},W_0)\ge
\dps\sum_{i=0}^{n}k^i-\dps\sum_{i=0}^{d}k^i=\dps\sum_{i=d}^{n}k^i.$$

However, $\displaystyle\lim_{t\to0}p_m^t\in W_0$ if only if $\widetilde{G}_i^0=0$
for $i=1,\ldots,d$. That is, $Q_1(\xi)=\ldots=Q_d(\xi)=0$ admits a
nontrivial solution. Thus, if the multivariate resultant
$Res(Q_1,\ldots,Q_d)=0$ then $W_0$ has an embedded closed point.
Given that all functions $Q_i$ are not identically null, the upper
bound of number of embedded points $W_0$ is $k+\ldots+k^{d-1}$.
Comparing with (\ref{ulte}), we obtain that

$$\dps\sum_{i=d}^{n}k^i\leq\mu(\mathcal{G},W_0)=\dps\sum_{i=d}^{n}k^i + q\leq \dps\sum_{i=1}^{n}k^i$$
where $q$ is the number of embedded closed points of $W_0$, for all
$k\ge 1$.
\vspace{1cm}

\paragraph{\bf {Acknowledgments.}}
We would like to acknowledge to Renato Vidal Martins (UFMG) and John
MacQuarrie (UFMG) for the suggestions and corrections to improve the
paper.


\begin{thebibliography}{99}

\bibitem{BB}
P. Baum, R. Bott: {\em On the zeros of meromorphic vector-fields.} Essays on Topology and Related topics, M\'emoires d\'de\'es \`a Georges de Rham, Springer-Verlag, Berlim, 1970, 29-47.
\bibitem{correa}
M. Corr\^ea JR, A. Fern\'andez-P\'erez, G. Nonato Costa, R. Vidal
Martins: \emph{Foliations by curves with curves as singularities}.
Annales de l'institut Fourier, Volume 64 (2014) no. 4, 1781-1805.

\bibitem{WF}
W. Fulton: {\em Intersection Theory.} Springer-Verlag Berlin Heidelberg, 1984.
\bibitem{GK}
X. Gomez-Mont, G. Kempf, \emph{Stability of meromorphic vector fields in projective spaces}. Comment. Math. Helv. \textbf{64 }(1989), 462-473.
\bibitem{GH}
P. Griffiths, J. Harris : {\em Principles of Algebraic Geometry.}
John Wiley\& Sons, Inc. 1994.
\bibitem{La}
R. Lazarsfeld:\emph{ Positivity in algebraic geometry, I, II},
Springer, 2004.
\bibitem{lins}
A. Lins Neto, M. G. Soares: {\em Algebraic solutions of one-dimensional foliations.} J. Differential Geom. 43 (1966), 652-673.
\bibitem{poincare}
H. Poincar\'e: {\em Sur l'Int\'egration Alg\'ebrique des \'Equations Differentielles du Premier Ordre et du Premier Degr\'e}. Rendiconti del Circolo Matematico di Palermo, 5 (1891), 161-191.
\bibitem{IP}
I.R. Porteous: {\em Blowing up Chern class.} Proc. Cambridge Phil.
Soc. {\bf 56} (1960), 118-124.
\bibitem{sancho}
F. Sancho Salas: {\em Number of singularities of a foliation on
$\P^n$.} Proceedings of the American Mathematical Society, {\bf 130}
(2001), 69-72
\bibitem{toulouse} Gilcione Nonato Costa: {\em Holomorphic foliations by
curves on $\mathbb{P}^{3}$ with non-isolated singularities.} Ann. Fac. Sci. Toulouse, Math. (6) 15, no. 2 (2006), 297-321.
\bibitem{indices} Gilcione Nonato Costa: {\em Indices Baum-Bott for curves of singularities.} Bulletin of the Brazilian Mathematical
Society, (2016), Volume 47, Issue 3, 883-910.
\bibitem{milnor}
J. Milnor: {\em Singular points of complex hypersurfaces.} Annals of
Mathematics Studies Vol. 61 Princeton Univ. Press, Princeton, N. J.,
1968.
\bibitem{soares}
Marcio G. Soares: {\em Projective varieties invariant by one-dimensional foliations}. Annals of Mathematics, 152 (2000), 369-382.
\bibitem{inventions}
Marcio G. Soares: {\em The Poincar\'e problem for hypersurfaces invariant by one-dimensional foliations.} Invent. Math. 128 (1997), 495-500.
\bibitem{bound}
Marcio G. Soares: {\em Bounding Poincar\'e-Hopf indices and Milnor numbers.} Math. Nachr. 278, no. 6, (2005), 703-711.
\bibitem{israel}
Israel Vainsencher: {\em Foliations singular along a curve}. Trans. London Math. Soc. 2 (1), (2015), 80-92.

\end{thebibliography}
\end{document}